\documentclass[11pt]{article}
\usepackage{amsmath,amsfonts,amsthm,amscd,amssymb,graphicx}
\usepackage{url}
\usepackage[utf8]{inputenc}
\usepackage{amsbsy}
\usepackage{graphicx, color}
\usepackage[text={7.0in,9.5in},centering,includefoot,foot=0.6in]{geometry}
\usepackage{verbatim}

\newtheorem{Lemma}{Lemma}[section]

\newtheorem{Definition}[Lemma]{Definition}

\newtheorem{Proposition}[Lemma]{Proposition}
\newtheorem{Remark}[Lemma]{Remark}
\newtheorem{Theorem}[Lemma]{Theorem}

 {\begin{trivlist}\item[]\textbf{Proof#1 }}%
 {\hspace*{\fill}$\rule{0.3\baselineskip}{0.35\baselineskip}$\end{trivlist}}

 {\begin{trivlist}\item[]\textbf{Acknowledgments }}{\end{trivlist}}

\makeatletter\@addtoreset{equation}{section}\makeatother

\def\Re{\mathop\mathrm{Re}\nolimits}    % real part
\def\Im{\mathop\mathrm{Im}\nolimits}    % imaginary part

\newcommand{\rmd}{\mathrm{d}}           % derivatives
\newcommand{\rmi}{\mathrm{i}}           % imaginary unit
\newcommand{\uu}{\mathbf{u}}

\begin{document}
\title{Selection of quasi-stationary states in the Navier-Stokes equation on the torus}
\author{Margaret Beck, Eric Cooper, Konstantinos Spiliopoulos\footnote{Department of Mathematics and Statistics, Boston University, Boston, MA 02215. E-mail: mabeck@bu.edu, cooper@bu.edu, kspiliop@math.bu.edu.}}
\date{}
\maketitle

\begin{abstract}
The two dimensional incompressible Navier-Stokes equation on $D_\delta := [0, 2\pi\delta] \times [0, 2\pi]$ with $\delta \approx 1$, periodic boundary conditions, and viscosity $0 < \nu \ll 1$ is considered.  Bars and dipoles, two explicitly given quasi-stationary states of the system, evolve on the time scale $\mathcal{O}(e^{-\nu t})$ and have been shown to play a key role in its long-time evolution. Of particular interest is the role that $\delta$ plays in selecting which of these two states is observed. Recent numerical studies suggest that, after a transient period of rapid decay of the high Fourier modes, the bar state will be selected if $\delta \neq 1$, while the dipole will be selected if $\delta = 1$. Our results support this claim and seek to mathematically formalize it.  We consider the system in Fourier space, project it onto a center manifold consisting of the lowest eight Fourier modes, and use this as a model to study the selection of bars and dipoles. It is shown for this ODE model that the value of $\delta$ controls the behavior of the asymptotic ratio of the low modes, thus determining the likelihood of observing a bar state or dipole after an initial transient period. Moreover, in our model, for all $\delta \approx 1$, there is an initial time period in which the high modes decay at the rapid rate $\mathcal{O}(e^{-t/\nu})$, while the low modes evolve at the slower $\mathcal{O}(e^{-\nu t})$ rate. The results for the ODE model are proven using energy estimates and invariant manifolds and further supported by formal asymptotic expansions and numerics.
\end{abstract}

%%%%%%%%%%%%%%

\section{Introduction}\label{intro}

In this paper we consider the 2D incompressible Navier-Stokes equation
\begin{equation}
\begin{aligned}
\partial_t \uu = \nu \Delta \uu &- (\uu \cdot \nabla) \uu - \nabla p \label{E:2dns} \\
\nabla \cdot \uu &= 0
\end{aligned}
\end{equation}
on the possibly asymmetric torus $(x,y) \in D_\delta := [0, 2\pi\delta] \times [0, 2\pi]$ with $\delta \approx 1$, periodic boundary conditions, and viscosity $0 < \nu \ll 1$. Defining $\omega = (0,0,1) \cdot (\nabla \times \uu)$, one obtains the 2D vorticity equation
\begin{equation}\label{E:2dvort}
\partial_t \omega = \nu \Delta \omega - \uu\cdot\nabla \omega, \qquad \uu = \begin{pmatrix} \partial_y(- \Delta^{-1}) \\ -\partial_x(- \Delta^{-1})\end{pmatrix} \omega.
\end{equation}
The relation between $\uu$ and $\omega$ is known as the Biot-Savart law. The periodic boundary conditions force $\int_{D_\delta} \omega = 0$, and hence $\Delta^{-1} \omega$ is well-defined.

 Because the viscosity is small, it is reasonable to expect that stationary solutions of the Euler equation (\eqref{E:2dns} or \eqref{E:2dvort} with $\nu = 0$) would play a role in the long-time evolution of the Navier-Stokes equation. However, the Euler equation has infinitely many stationary solutions, so it is not obvious which such solutions are important.  In \cite{Yin}, entropy arguments and extensive numerical studies were conducted in the case $\delta = 1$ and suggested that the so-called bar states and dipoles should be the two most important stationary solutions of the Euler equations. Although both states were observed after initial transient periods in the evolution of the Navier-Stokes equation, interestingly the dipole seemed to emerge for a large class of initial data, whereas the bar state only emerged for a special class of initial data.  A later study \cite{BouchetSimonnet09} numerically analyzed \eqref{E:2dns} on $D_\delta$ with the addition of a certain type of stochastic forcing. There, after an initial transient period, a metastable switching between the bars and dipoles was seen, with the dipole being dominant for $\delta = 1$ and the bar states being dominant for $\delta \neq 1$. Related analytical work was conducted in \cite{BeckWayne13, IbrahimMaekawaMasmoudi17} where the rate of convergence to a bar state for appropriate initial conditions was shown to be $\mathcal{O}(e^{-\sqrt{\nu}t})$, while the bar state itself decayed at the  $\mathcal{O}(e^{-\nu t})$ background rate.  In this work, we will analyze the selection of bars and dipoles,  based on the parameter $\delta$.  At this point, we also refer the interested reader to \cite{Armbuster,FoiasLuanSaut,FoiasSaut2,FoiasSaut1,GalletYoung,KimOkamodo} for steady state results and results in the asymptotic regime as time goes to infinity.

If $\delta = 1$, any function of the form
\begin{equation}\label{E:family}
\omega(x,y;m) = e^{-\frac{\nu m^2}{\delta^2} t}[ a_1 \cos(mx/\delta) + a_2 \sin (mx/\delta)] + e^{-\nu m^2 t}[a_3 \cos(my) + a_4 \sin (my) ], \qquad m \in \mathbb{Z}.
\end{equation}
is an exact solution to \eqref{E:2dvort}. If $\delta \neq 1$, then (\ref{E:family}) is an exact solution to \eqref{E:2dvort} if and only if $a_1 = a_2 = 0$ or if $a_3 = a_4 = 0$. Bar states, also known as unidirectional or Komogorov flow, are members of this family for $m = 1$ given by
\[
\omega_{bar}(x,t) = e^{-\frac{\nu}{\delta^2}t} \sin (x/\delta), \qquad \omega_{bar}(y ,t) = e^{-\nu t} \sin y,
\]
or similarly with sine replaced by cosine. The associated velocity fields are given by
\[
\uu_{bar}(x,t) = -\delta e^{-\frac{\nu}{\delta^2}t} \left(
\begin{array}{c}
0\\
\cos(x/\delta)\\
\end{array}
\right), \qquad \uu_{bar}(y ,t) = e^{-\nu t} \left(
\begin{array}{c}
\cos y\\
0\\
\end{array}
\right),
\]
respectively. The dipoles are also members of the family for $m = 1$ and are given by
\[
\omega_{dipole}(x,y,t) = e^{-\frac{\nu}{\delta^2} t} \sin (x/\delta) + e^{-\nu t}\sin y,
\]
with velocity field
\[
\uu_{dipole}(x,y,t) =  \left(
\begin{array}{c}
e^{-\nu t}\cos y\\
-\delta e^{-\frac{\nu}{\delta^2}t}\cos(x/\delta)\\
\end{array}
\right),
\]
or similarly with sine replaced by cosine.  The bar states are exact solutions of \eqref{E:2dvort} for all $\delta \approx 1$, while the dipoles are only exact solutions for $\delta = 1$. In addition to the references mentioned above, the bar states were also studied analytically in \cite{MeshalkinSinaui61}. Although the setting was slightly different, their results suggest that, when $\delta =1$ an $m$-bar state $e^{-\nu m^2 t}\cos (my)$ (or similarly with sine replaced by cosine or $y$ replaced by $x$) is attracting if and only if $m = 1$. Because the dipoles are only approximate solutions for $\delta \neq 1$, it may be intuitive that they would not play a key role in the long-time evolution in that case. However, they were still observed in the metastable switching in the appropriately stochastically forced Navier-Stokes equation for $\delta \neq 1$ \cite{BouchetSimonnet09}.

Because of the form of the bar states and dipoles, it is useful to study \eqref{E:2dvort} in Fourier space, in which it can be written
\begin{equation}
\begin{aligned}
\dot{\hat{\omega}}_{\vec{k}} &= -\frac{\nu}{\delta^2} |\vec{k}|^2_\delta \hat{\omega}_{\vec{k}} - \delta \sum_{\vec{l}} \frac{\langle \vec{k}^\perp, \vec{l}\rangle}{|\vec{l}|_\delta^2} \hat{\omega}_{\vec{k}-\vec{l}}\hat{\omega}_{\vec{l}} \\
&= -\frac{\nu}{\delta^2} |\vec{k}|^2_\delta \hat{\omega}_{\vec{k}} - \frac{\delta}{2} \sum_{\vec{j}+\vec{l}=\vec{k}} \langle \vec{j}^\perp, \vec{l}\rangle \left( \frac{1}{|\vec{l}|_\delta^2} -  \frac{1}{|\vec{j}|_\delta^2} \right) \hat{\omega}_{\vec{j}}\hat{\omega}_{\vec{l}}, \label{E:vort}
\end{aligned}
\end{equation}
where
\[
|\vec{k}|_\delta^2 = k_1^2 + \delta^2 k_2^2, \qquad \vec{k}^\perp = (k_2, -k_1)
\]
and
\[
\omega(x,y) = \sum_{\vec{k} \neq 0} \hat{\omega}_{\vec{k}} e^{\rmi (k_1x/\delta + k_2y)}, \qquad \hat{\omega}_{\vec{k}} = \frac{1}{4\pi^2\delta}\int_D \omega(x,y) e^{-\rmi(k_1x/\delta + k_2y)} \rmd x \rmd y.
\]
In terms of these variables, the $y$-bar states $e^{-\nu t} \cos y$ and $e^{-\nu t} \sin y$ correspond to solutions with energy only in the $\vec{k} = (0, \pm 1)$ modes, the $x$-bar states $e^{-\frac{\nu}{\delta^2} t} \cos (x/\delta)$ and $e^{-\frac{\nu}{\delta^2} t} \sin (x/\delta)$ correspond to solutions with energy only in the $\vec{k} = (\pm 1, 0)$ modes, and the dipoles correspond to solutions with energy in both the $\vec{k} = (0, \pm 1)$ and $\vec{k} = (\pm 1, 0)$ modes. These four modes are the lowest modes in the system, in that they correspond to the modes with the lowest values of $|\vec{k}|_\delta$, with $|\vec{k}|_\delta = 1$ or $\delta^2$. We will refer to any modes with $|\vec{k}|_\delta > \max\{1, \delta^2\}$ to be high modes.

 When $\delta = 1$, the set $\{ \hat \omega_{\vec{k}} = 0 \mbox{ if } |\vec{k}| > 1\}$ is an exact global invariant manifold for \eqref{E:vort}. However, the dynamics on it are trivial, determined by the linear terms. Therefore, even though both the bars and dipoles lie within this manifold, if we want to understand how the system selects between them, we must include at least some of the higher modes. To do so, we conduct a center manifold reduction on \eqref{E:vort} and project onto the lowest eight modes, which we denote by
\begin{eqnarray}
\omega_1 &:=& \hat\omega(1,0),\quad \omega_2 := \hat\omega(-1, 0),\quad \omega_3 := \hat\omega(0, 1), \quad \omega_4 := \hat\omega(0, -1), \nonumber \\
\omega_5 &:=& \hat\omega(1,1),\quad \omega_6 := \hat\omega(-1, 1), \quad \omega_7 := \hat\omega(1, -1), \quad  \omega_8 := \hat\omega(-1, -1). \label{E:fourier}
\end{eqnarray}
The variables $\omega_{1, 2, 3, 4}$ correspond to the low modes, while $\omega_{5, 6, 7, 8}$ represent the role of all the high modes. Since the solution $\omega(x,y)$ of \eqref{E:2dvort} is real valued,
\begin{equation}\label{symmetry}
\omega_1 = \bar \omega_2, \quad \omega_3 = \bar \omega_4, \quad \omega_5 = \bar \omega_8, \quad \omega_7 = \bar \omega_6.
\end{equation}
Thus, the resulting ODE, which is derived in detail in \S\ref{CM}, will be eight dimensional.

The reduction to the eight-dimensional ODE is local. In fact, since the size of the spectral gaps for the linear operator, $\nu \Delta$, is $\mathcal{O}(\nu)$, this reduction will only be valid in a small neighborhood of $0$ of size $\mathcal{O}(\nu)$. Moreover, one cannot expect to obtain a finite-dimensional model of the full system \eqref{E:vort} that describes the global dynamics \cite{Zelik14}. However, we will still use this finite-dimensional model to provide insight into the potential role that $\delta$ plays in the selection of bars and dipoles. For other examples in which finite-dimensional models have been used to study the dynamics of the Navier-Stokes equation, see \cite{EMattingly01, MattinglyPardoux14}.

The ODE derived in \S\ref{CM} will be analyzed in \S\ref{SymmetricTorus}-\S\ref{GeomSingPert}. In \S\ref{SymmetricTorus} we focus on the case $\delta = 1$, which corresponds to the symmetric torus, and in \S\ref{AsymmetricTorus} we focus on the case $\delta \neq 1$. In both cases, to study the relative importance of the bar states versus the dipoles, we consider the evolution of the ratio $R(t):=\frac{|\omega_1(t)|^2}{|\omega_3(t)|^2}$. (In the case $\omega_3 = 0$, one can study the inverse of this quantity.) Note that asymptotic convergence of $R(t) \to 0$ or $\infty$ would correspond to convergence to a $y$- or $x$-bar state, respectively, while convergence to some finite, nonzero value would correspond to convergence to a dipole. (We note, however, that if $R(t) \to R_\infty \ll 1$ for example, then such a state would qualitatively appear to be a $y$-bar state, even though there would be nonzero variation in $x$.)

For the case $\delta = 1$, in Theorem \ref{Symmetrictheorem} we show that there is a family of co-dimension one stable manifolds in the phase space of the ODE that determines the asymptotic limit of $R(t)$. The limit $R(t) \to 0$ corresponds to exactly one of these manifolds, and hence a $y$-bar state would only be observed for the special class of initial conditions starting on this manifold. (A similar result holds for the $x$-bar states.) Therefore, we conclude that, for the symmetric torus, general initial conditions will typically lead to the emergence of a dipole as the dominant quasi-stationary state. For the case $\delta \neq 1$, the single center direction that had been present in the system for $\delta = 1$ becomes hyperbolic, with the sign of $\delta - 1$ determining if it is expanding or contracting. Thus, this selects the limit $R(t) \to 0$ or $R(t) \to \infty$, selecting an $y$-bar state or $x$-state respectively. These results are found in Theorem \ref{Asymmetrictheorem}.

In both cases, $\delta = 1$ and $\delta \neq 1$, we additionally show that the high modes decay at the rate $\mathcal{O}(e^{-t/\nu})$, while the low modes decay at the rate $\mathcal{O}(e^{-\nu t})$. These results can be found in Lemma \ref{lem:sym-decay} and Proposition \ref{Brapiddecay}. This allows for the rapid convergence to a metastable state as seen for the bar states in \cite{BeckWayne13, IbrahimMaekawaMasmoudi17}. We note however, that the rapid decay in \cite{BeckWayne13, IbrahimMaekawaMasmoudi17} was $\mathcal{O}(e^{-\sqrt{\nu}t})$, whereas here we obtain decay of the high modes at $\mathcal{O}(e^{-t/\nu})$. We expect that this discrepancy is due to the fact that the ODE we study is only a model of the full PDE. The main point is the qualitative prediction of a separation in time scales, rather than the specific rate. In addition, we comment here that the reason we do not scale the $\nu$ out of the equation

In \S\ref{GeomSingPert}, we reframe the problem as a perturbation problem, so as to confirm the results of the earlier sections using a different method. After setting $\delta=1+\epsilon_0\epsilon$, where $\epsilon_0= \pm 1$ will determine if $\delta$ is less than or greater than 1, relating the parameters $\nu$ and $\delta$ via $\epsilon$,  and scaling the system in an appropriate way, a slow-fast system emerges. Perturbation expansions are then used to
illustrate the claims made in \S\ref{SymmetricTorus}  and \S\ref{AsymmetricTorus}. These expansions also reveal that evolution to a bar state accelerates as $\delta$ is moved slightly farther from 1. This result is consistent with the work of \cite{BouchetSimonnet09} in that the simulations done there suggest that a bar state dominates the metastable stochastic transitions only when $\delta$ is sufficiently far from 1.

Lastly, we conclude the introduction by pointing out that one could in principle scale time and velocity from the very beginning in order to make $\nu=1$. However, our goal here is to investigate the selection mechanism based on the value of $\delta$ and it turns out that there is an interplay with $\nu$ as well. As a matter of fact, estimates like those in Lemma \ref{lem:sym-decay} and Proposition \ref{Brapiddecay} depend on $\nu$ in very precise ways and if one had scaled out $\nu$ in the beginning, one would have to undo the scalings later on. Hence, we chose to keep the dependence on $\nu$ as it is given by the equation originally.

%%%%%%%%%%%%%%%%%%%%%%

\section{Center Manifold}\label{CM}

 In this section, we carry out a center manifold reduction of \eqref{E:vort} onto the eight modes listed in \eqref{E:fourier}. This is a standard calculation that can be found, for example, in \cite{Henry81}.

The basic idea is, for any $\omega_{\vec{k}}$ with $\vec{k} \notin \{(\pm1, 0), (0, \pm1), (\pm1, \pm1)\} =: \mathcal{K}_0$, to assume that there exists a smooth function $H(\omega_1, \dots, \omega_8; \vec{k})$ such that the eight-dimensional manifold defined by

\[
\mathcal{M} = \{ \hat{\omega}: \hat{\omega}_{\vec{k}} = H(\omega_1, \dots, \omega_8; \vec{k}), \quad \vec{k} \notin \mathcal{K}_0\}
\]
is invariant for the dynamics of \eqref{E:vort}. We refer to this as a center manifold because it is defined in terms of the lowest eight modes, which have the weakest linear decay rates.  Based on this assumption, one can then in principle compute the coefficients of the
Taylor expansion of $H(\cdot, \vec{k})$ to any order, for each $\vec{k}$. To make this precise, define
\begin{equation}
    \begin{aligned}
\hat{\omega}_{\vec{k}} &= (a_1\omega_1^2 + \dots + a_8 \omega_8^2) + (b_{12}\omega_1\omega_2 + \dots b_{18}\omega_1\omega_8) + (c_{23}\omega_2\omega_3 + \dots c_{28}\omega_2\omega_8) \\
&\quad + (d_{34}\omega_3\omega_4 + \dots + d_{38}\omega_3 \omega_8) + (e_{45}\omega_4\omega_5 + \dots e_{48}\omega_4\omega_8) + (f_{56} \omega_5 \omega_6 + \dots + f_{58}\omega_5\omega_8)  \\
& \quad + (g_{67}\omega_6\omega_7 + g_{68}\omega_6\omega_8) + h_{78}\omega_7\omega_8 +\Theta(3)  \\
&=: H(\omega_1, \dots, \omega_8; \vec{k}), \label{E:analytic}
    \end{aligned}
\end{equation}
where $\Theta(3)$ is defined in Definition \ref{Def:ThetaOrder}.
\begin{Definition}\label{Def:ThetaOrder}
Let $n\in\mathbb{N}$ and let $\omega_1,\dots,\omega_8$ be defined as in \eqref{E:fourier}. We define $\Theta(n)$ to be the set of terms that are of the form $\omega_1^{n_1}\omega_2^{n_2}...\omega_8^{n_8}$ where $n_1+...+n_8\geq n$ for $n_1, \dots, n_8 \in \mathbb{N}$. Thus, $\Theta(n)$ is the set of monomials in $\omega_1,\dots,\omega_8$ of degree $n$ or larger.
\end{Definition}

In order to determine the values of the unknown coefficients that appear in \eqref{E:analytic}, we will compute $\partial_t \omega_j$, for $j = 1, \dots, 8$ in two ways: 1) by using the right hand side of equation \eqref{E:vort}, where we substitute $H(\cdot, \vec{k})$ in for $\omega_{\vec{k}}$ whenever $\vec{k} \notin \mathcal{K}_0$; and 2) by computing the time derivative of the expansion in \eqref{E:analytic}, substituting in the equation for $\dot{\omega}_j$, $j = 1, \dots, 8$ given in \eqref{E:vort} as necessary. Equating these two results, and comparing terms with equal order in powers of $\omega_1, \dots, \omega_8$, will lead to equations that should in theory determine the values of the coefficients. Note that the coefficients in \eqref{E:analytic} will depend on $\vec{k}$, but we have suppressed this for notational convenience.

We are only interested in computing the expansion for $H(\cdot, \vec{k})$ up to and including terms of $\Theta(2)$. Therefore, when carrying out the above-described calculation, we will only need to retain terms up to any including $\Theta(3)$. This means any product of the form $\omega_{\vec{k}}\omega_{\vec{j}}$ with $\vec{k}, \vec{j} \notin \mathcal{K}_0$ will be of higher order, and thus we can discard it. Carrying out step 1) above, leads to
\begin{equation}
    \begin{aligned}
        \dot \omega_1 &= -\frac{\nu}{\delta^2}\omega_1 + \frac{1}{\delta(1+\delta^2)}[ \omega_3 \omega_7 - \bar{\omega}_3 \omega_5] + \frac{3\delta}{(4+\delta^2)(1+\delta^2)}[\hat\omega(2,-1)\bar{\omega}_7-\hat\omega(2,1)\bar{\omega}_5] + \Theta(4)   \\
        \dot \omega_3 &= - \nu \omega_3 + \frac{\delta^3}{1+\delta^2}[\bar{\omega}_1\omega_5 - \omega_1\bar{\omega}_7]+ \frac{3\delta^3}{(1+4\delta^2)(1+\delta^2)}[\hat\omega(1,2)\bar{\omega}_5-\hat\omega(-1,2)\omega_7] + \Theta(4) \label{E:lowest4}  \\
        \dot \omega_5 &= -\frac{\nu}{\delta^2}(1+\delta^2)\omega_5 - \frac{(\delta^2-1)}{\delta}\omega_1\omega_3 + \frac{\delta(3+\delta^2)}{4+\delta^2}\hat\omega(2,1)\bar{\omega}_1 -\frac{1+3\delta^2}{\delta(1+4\delta^2)}\hat\omega(1,2)\bar{\omega}_3    \\
         & \quad +\frac{\delta(3-\delta^2)}{2(1+\delta^2)}\hat{\omega}(2,0)\bar{\omega}_7 + \frac{1-3\delta^2}{2\delta(1+\delta^2)}\hat{\omega}(0,2)\omega_7 + \Theta(4)   \\
        \dot \omega_7 &= -\frac{\nu}{\delta^2}(1+\delta^2)\omega_7 + \frac{(\delta^2-1)}{\delta}\omega_1\bar{\omega}_3 + \frac{1+3\delta^2}{\delta(1+4\delta^2)}\hat\omega(1,-2)\omega_3 -\frac{\delta(3+\delta^2)}{4+\delta^2}\hat\omega(2,-1)\bar{\omega}_1   \\
        & \quad + \frac{3\delta^2-1}{2\delta(1+\delta^2)}\hat{\omega}(0,-2)\omega_5 + \frac{\delta(\delta^2-3)}{2(1+\delta^2)}\hat{\omega}(2,0)\bar{\omega}_5 + \Theta(4)
    \end{aligned}
\end{equation}
Note we have listed only four of the equations, due to \eqref{symmetry}. Therefore, we need only focus on determining the coefficients of $H(\cdot, \vec{k})$ for $\vec{k} \in \mathcal{K} := \{(\pm1, \pm 2), (\pm 2, \pm1), (\pm 2, 0), (0, \pm2)\}$.
Carrying out step 2) leads to
\begin{equation}\label{E:method2}
\dot \omega_{\vec{k}} = \nabla H(\omega_1, \dots, \omega_8; \vec{k}) \cdot \left[ -\nu \left(\frac{\omega_1}{\delta^2},\frac{\omega_2}{\delta^2},\omega_3, \omega_4, \frac{1+\delta^2}{\delta^2}\omega_5, \dots, \frac{1+\delta^2}{\delta^2}\omega_8\right) + \Theta(2) \right].
\end{equation}
Equating \eqref{E:lowest4} and \eqref{E:method2} leads to, for example, the following for the $b_{15}$ coefficient of $H(\omega_1, \dots, \omega_8; (2, 1))$:
\[
  - \frac{\nu}{\delta^2} b_{15} -\frac{\nu}{\delta^2}(1+\delta^2) b_{15} = -\frac{\nu}{\delta^2}(4+\delta^2) b_{15} - \frac{\delta^3}{1+\delta^2} \quad \Rightarrow \quad b_{15} = -\frac{\delta^5}{2\nu(1+\delta^2)}.
\]
Continuing in this manner, we find
\begin{equation}
    \begin{aligned}
        \hat\omega(2,1) &= - \frac{\delta^5}{2\nu(1+\delta^2)} \omega_1 \omega_5 + \Theta(3) \\
        \hat\omega(2,-1) &=  \frac{\delta^5}{2\nu(1+\delta^2)} \omega_1 \omega_7 + \Theta(3) \\
        \hat\omega(1,2) &=  \frac{1}{2\nu\delta(1+\delta^2)} \omega_3\omega_5 + \Theta(3) \label{E:highermodes} \\
        \hat\omega(1,-2) &=  -\frac{1}{2\nu\delta(1+\delta^2)} \bar{\omega}_3\omega_7 + \Theta(3).
    \end{aligned}
\end{equation}
Conveniently, most of the coefficients are zero. In order for the above equations to be the unique expansion for the $H$'s that satisfy the invariance condition, we need to restrict the values of $\delta$ to be sufficiently close to $1$. To see why this is the case, consider, for example, the $a_1$ coefficient of $\vec{k}=(2,1)$. Its defining equation is given by
\[
-2\nu\frac{(1+\delta^2)}{\delta^2}a_1 = -\nu\frac{(4+\delta^2)}{\delta^2}a_1
\]
If $\delta = \sqrt{2}$, then $a_1 = 0$ is not the unique solution. To prevent such ambiguities and obtain a unique expansion for
$\hat{\omega}(2,1)$, $\hat{\omega}(2,-1)$, $\hat{\omega}(1,2)$, and $\hat{\omega}(1,-2)$, it turns out that we should restrict the value of $\delta$ to
$(\frac{1}{\sqrt{2}},\sqrt{2})$.

Computing the expansions for $\hat{\omega}(\pm 2,0)$ and $\hat{\omega}(0,\pm 2)$ turns out to be somewhat different.
When $\delta \neq1$, one can check that, if $\sqrt{\frac{2}{3}}<\delta<\sqrt{\frac{3}{2}}$, the coefficients are all unique and equal to 0. However, when $\delta=1$, a unique set of coefficients cannot be determined. In this case, for any constants $\alpha_i$ and $\beta_i$, the following functions will satisfy the invariance condition
\begin{equation}
    \begin{aligned}
        \hat{\omega}(2,0)=G_{(2,0)}(\omega_5,\omega_7)&:=\alpha_1|\omega_5|^2+\alpha_2|\omega_7|^2+\alpha_3(\omega_5)^2 +\alpha_4(\bar{\omega}_5)^2+\alpha_5(\omega_7)^2+\alpha_6(\bar{\omega}_7)^2 \\
        & \qquad+\alpha_7\omega_5\omega_7+\alpha_8\omega_5\bar{\omega}_7+\alpha_9\bar{\omega}_5\omega_7+\alpha_{10}\bar{\omega}_5\bar{\omega}_7 + \Theta(3) \label{E:G} \\
        \hat{\omega}(0,2)=G_{(0,2)}(\omega_5,\omega_7)&:=\beta_1|\omega_5|^2+\beta_2|\omega_7|^2+\beta_3(\omega_5)^2 +\beta_4(\bar{\omega}_5)^2+\beta_5(\omega_7)^2+\beta_6(\bar{\omega}_7)^2  \\
        & \qquad+\beta_7\omega_5\omega_7+\beta_8\omega_5\bar{\omega}_7+\beta_9\bar{\omega}_5\omega_7+\beta_{10}\bar{\omega}_5\bar{\omega}_7 + \Theta(3)
    \end{aligned}
\end{equation}
Note that we have relabeled these functions using the letter $G$, rather than $H$ as above, to emphasize that this is a special case only when $\delta = 1$. Moreover, we have used \eqref{symmetry} to write these as functions of $\omega_{5,7}$ only, for notational convenience and to highlight the fact that they depend only on the high modes. Similar equations for $\hat{\omega}(-2, 0)$ and $\hat{\omega}(0, -2)$ can be found, based on \eqref{E:G}, using \eqref{symmetry}. We will comment more on this issue of nonuniqueness in the $\delta = 1$ case in Remark \ref{R:nonuniqueness}, below. Thus, we arrive at the following eight-dimensional ODE model, for $\delta\neq1$ and $\sqrt{\frac{2}{3}}<\delta<\sqrt{\frac{3}{2}}$
\begin{equation}
    \begin{aligned}
        \dot \omega_1 &= - \frac{\nu}{\delta^2} \omega_1 + \frac{1}{\delta(1+\delta^2)}[\omega_3\omega_7 - \bar{\omega}_3 \omega_5] + \frac{3\delta^6}{2\nu(4+\delta^2)(1+\delta^2)^2}\omega_1(|\omega_5|^2 + |\omega_7|^2)  \\
        \dot \omega_3 &= - \nu \omega_3 + \frac{\delta^3}{(1+\delta^2)}[\bar{\omega}_1\omega_5 - \omega_1 \bar{\omega}_7] + \frac{3\delta^2}{2\nu(1+4\delta^2)(1+\delta^2)^2}\omega_3(|\omega_5|^2 + |\omega_7|^2) \label{E:deltanot1}  \\
        \dot \omega_5 &= - \nu \frac{1+\delta^2}{\delta^2} \omega_5 -\frac{\delta^2-1}{\delta}\omega_1 \omega_3   -\frac{\delta^6(3+\delta^2)}{2\nu(4+\delta^2)(1+\delta^2)} \omega_5 |\omega_1|^2-\frac{1+3\delta^2}{2\nu\delta^2(1+4\delta^2)(1+\delta^2)}\omega_5 |\omega_3|^2    \\
        \dot \omega_7 &= - \nu \frac{1+\delta^2}{\delta^2} \omega_7 +\frac{\delta^2-1}{\delta}\omega_1 \bar{\omega}_3 -\frac{\delta^6(3+\delta^2)}{2\nu(4+\delta^2)(1+\delta^2)} \omega_7 |\omega_1|^2 - \frac{1+3\delta^2}{2\nu\delta^2(1+4\delta^2)(1+\delta^2)}\omega_7 |\omega_3|^2
    \end{aligned}
\end{equation}
and for $\delta = 1$ we find the following eight-dimensional ODE model
\begin{equation}
     \begin{aligned}
         \dot\omega_1 &= -\nu \omega_1 +\frac{1}{2}(\omega_3\omega_7-\bar{\omega}_3\omega_5) + \frac{3}{40\nu}\omega_1(|\omega_5|^2+|\omega_7|^2)  \\
        \dot\omega_3 &= -\nu \omega_3 +\frac{1}{2}(\bar{\omega}_3\omega_5-\omega_1\bar{\omega}_7) + \frac{3}{40\nu}\omega_3(|\omega_5|^2+|\omega_7|^2) \label{E:delta=1}\\
        \dot\omega_5 &= -2\nu \omega_5 - \frac{1}{5\nu}\omega_5(|\omega_1|^2+|\omega_3|^2) +\frac{1}{2} (\bar{\omega}_7 G_{(2,0)}-\omega_7 G_{(0,2)}) \\
        \dot\omega_7 &= -2\nu \omega_7 - \frac{1}{5\nu}\omega_7(|\omega_1|^2+|\omega_3|^2) +\frac{1}{2} (\omega_7 \bar{G}_{(0,2)}-\bar{\omega}_5 G_{(2,0)})
    \end{aligned}
\end{equation}

\begin{Remark}\label{R:nonuniqueness}
We suspect that the nonuniqueness of the expansions in the $\delta = 1$ case is related to the fact that the above calculation is only local and valid in a small neighborhood of $\mathcal{O}(\nu)$. In order to obtain continuity of our ODE model for $\delta \approx 1$, it makes sense to chose $G_{(0,2)} = G_{(2,0)} = 0$ in the $\delta = 1$ case. This makes \eqref{E:deltanot1} equal to \eqref{E:delta=1} in the limit $\delta \to 1$. Therefore, we make this choice in the following sections, and consider only equation \eqref{E:deltanot1} for all $\delta \in  (\sqrt{\frac{2}{3}}, \sqrt{\frac{3}{2}})$.
\end{Remark}

We conclude this section with a quick observation that reflects a symmetry in \eqref{E:vort} and that can be used to simplify some of the proofs in the following sections. Its proof is omitted as it is relatively straighforward.
\begin{Lemma} \label{real}
For any $\delta$, the set $\{\Im(\omega_1)=\Im(\omega_3)=\Im(\omega_5)=\Im(\omega_7)=0\}$ is invariant for \eqref{E:deltanot1}.
\end{Lemma}

%%%%%%%%%%%%%%%%%%%%%%%%%%%%%%%%%%%%%%%%%

\section{Symmetric Torus}\label{SymmetricTorus}

We now focus on \eqref{E:deltanot1} with $\delta = 1$, which corresponds to the symmetric torus. The goal will be to prove two results: Lemma \ref{decaydelta=1}, which states that the high modes decay much more rapidly than the low modes, and Theorem
\ref{Symmetrictheorem}, which states that most initial conditions will evolve toward a dipole, rather than a bar state.
We begin by rewriting system \eqref{E:deltanot1} with $\delta=1$:
\begin{equation}
    \begin{aligned}
        \dot \omega_1 &= - \nu \omega_1 + \frac{1}{2}[\omega_3\omega_7 - \bar{\omega}_3 \omega_5] + \frac{3}{40\nu}\omega_1(|\omega_5|^2 + |\omega_7|^2)  \\
        \dot \omega_3 &= - \nu \omega_3 + \frac{1}{2}[\bar{\omega}_1\omega_5 - \omega_1 \bar{\omega}_7] + \frac{3}{40\nu}\omega_3(|\omega_5|^2 + |\omega_7|^2) \label{E:symmetric} \\
        \dot \omega_5 &= - 2\nu \omega_5 - \frac{1}{5\nu}\omega_5(|\omega_1|^2+|\omega_3|^2) \\
        \dot \omega_7 &= - 2\nu \omega_7 - \frac{1}{5\nu}\omega_7(|\omega_1|^2+|\omega_3|^2)
    \end{aligned}
\end{equation}
It will be helpful to study the evolution of the quantities $A:=|\omega_1|^2+|\omega_3|^2$ and $B:=|\omega_5|^2+|\omega_7|^2$, in order to separate the evolution of the low modes from that of the high modes, and the quantity $R(t) = \frac{|\omega_1(t)|^2}{|\omega_3(t)|^2}$ to study whether it is a bar state or dipole that can be expected to be observed for large time. Recall that $R(t) \to 0$ corresponds to evolution towards a $y$-bar state, $R(t) \to \infty$ corresponds to evolution to an $x$-bar state, and $R(t) \to R_\infty \in (0, \infty)$ corresponds to convergence to a dipole, as $t\to\infty$.

Simulations of \eqref{E:symmetric} shown in Figure \ref{figAB} below suggest that there is a separation in time scales between the evolution of $A$ and $B$, consistent with previous numerical studies of the full Navier-Stokes equation \cite{Yin}. Moreover, $A$ exhibits an initial period of growth, before it begins to decay. Furthermore, in Figure \ref{figR} we see that a variety of initial values for $R$ lead to solutions that converge to a dipole. These behaviors will be made more precise below, in Lemma \ref{lem:sym-decay} and Theorem \ref{Symmetrictheorem}.
\begin{figure}[h!]
	\centering
	\includegraphics[width=10cm, height=6cm]{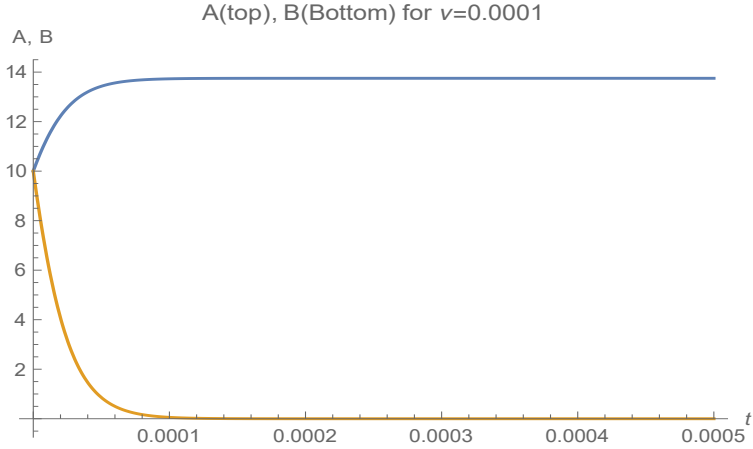}
	\caption{Rapid decay of higher Fourier modes.}
	\label{figAB}
\end{figure}

\begin{figure}[h!]
	\centering
	\includegraphics[width=10cm, height=6cm]{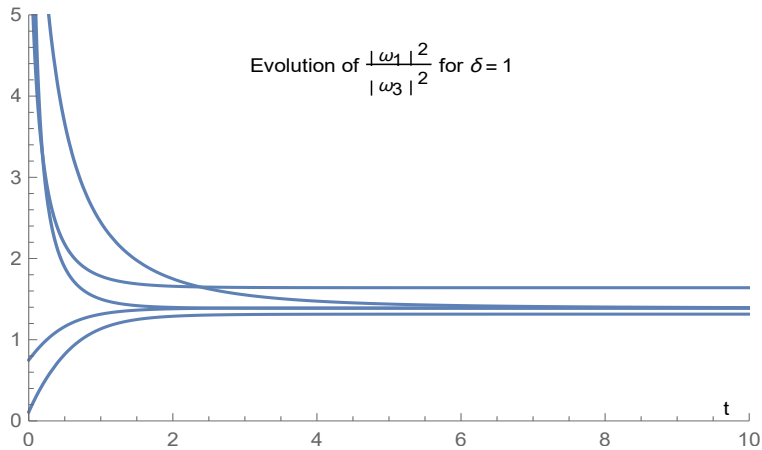}
	\caption{Initial conditions that evolve to a dipole state.}
	\label{figR}
\end{figure}

\begin{Lemma} \label{lem:sym-decay}
\label{decaydelta=1}
	Define  $A(t)$ := $|\omega_1(t)|^2+|\omega_3(t)|^2$ and $B(t):=|\omega_5(t)|^2+|\omega_7(t)|^2$. Let $t_0 = 1/\nu$, $\delta = 1$, and denote the initial data by $A(0) = A_0$ and $B(0) = B_0$. We have
\[
A(t) +B(t) \leq (A_0+B_0)e^{-2\nu t} \qquad \mbox{for all} \qquad t \geq 0.
\]
Moreover, for all $0 \leq t \leq t_0$, $A(t) \geq A_0 e^{-2}$ and $B(t) \leq B_0 e^{-\frac{2A_0}{5\nu e^2}t}$. Finally, for all $t \geq t_0$, $B(t) \leq B_0 e^{-\frac{2A_0}{5\nu^2 e^2}}$.
\end{Lemma}
\begin{proof}
Using \eqref{E:symmetric} we find that the dynamics of $A$ and $B$ are governed by
\begin{equation}
    \begin{aligned}
        \dot A &= -2\nu A +\frac{3}{20\nu}AB \label{E:AB} \\
        \dot B &= -4\nu B - \frac{2}{5\nu}AB
    \end{aligned}
\end{equation}
The first claim follows from the fact that, since $A$ and $B$ are both nonnegative,
\[
\frac{d}{dt}(A + B) = -2\nu (A+B) - 2\nu B - \frac{1}{4\nu}AB \leq - 2\nu(A+B).
\]
Also, since $A$ and $B$ are nonnegative, $\dot A \geq -2\nu A$, and so $A(t) \geq A_0e^{-2\nu t}$. As a result, for all $0 \leq t \leq t_0$, $A(t) \geq A_0 e^{-2}$. With this estimate, we then see that for all $0 \leq t \leq t_0$
\[
\dot B \leq - \left( 4\nu + \frac{2A_0}{5\nu e^2}\right) B \leq - \frac{2A_0}{5\nu e^2}B,
\]
from which the claim about $B$ during this time period follows by Gronwall's Inequality. Finally, note that $\dot B \leq 0$ so $B$ is monotonically decreasing (in fact strictly decreasing as long as $B > 0$), which proves the final claim.
\end{proof}

\begin{Remark}
Regarding the timescales predicted by the above lemma, recall that the center manifold reduction is expected to be valid only in an $\mathcal{O}(\nu)$ neighborhood of $\omega = 0$. If we correspondingly scale $A = \nu^2 \tilde A$ and $B = \nu^2 \tilde B$, then \eqref{E:AB} becomes
\begin{eqnarray*}
\dot{\tilde A} &= -2\nu \tilde A +\frac{3\nu}{20}\tilde A \tilde B \\
\dot{\tilde B} &= -4\nu \tilde B - \frac{2\nu}{5}\tilde A \tilde B.
\end{eqnarray*}
%which appears to remove the timescale separation.
Note, however, that despite this rescaling with $\nu$, and due to the difference in sign of the nonlinear terms, there will  (for most initial data) be an transient period when $\tilde A$ grows, while $\tilde B$ decays, before both $\tilde A$ and $\tilde B$ decay to zero.  We thank an anonymous referee for pointing this out.
\end{Remark}

This proposition shows that, when $\nu$ is small, for a long $\mathcal{O}(1/\nu)$ transient period, $B(t)$ is decaying at the rapid rate $\mathcal{O}(e^{-t/\nu})$, while $A(t)$ is not changing much. After this transient period, $B$ has become exponentially small and both $A$ and $B$ decay to zero at the background rate, $\mathcal{O}(e^{-2\nu t})$. Figures \ref{figlogA} and \ref{figlogB} below show the evolution of $A(t)$ and $B(t)$ on a logarithmic scale. These plots elucidate the results of Lemma \ref{lem:sym-decay}. We see that following the initial period of growth, the curve L=$\log(A(t))$ decreases at a rate very close to $L=-2\nu t$, suggesting that the low modes indeed decay no faster than the global background decay rate. The particular line, $L=\max(A(t))-2\nu t$, that was graphed was an asthetic choice to clearly show the two curves become parallel. The trend can be observed to continue for longer times. Figure \ref{figlogB} illustrates that for $0<t<\frac{1}{\nu}=100$, $L=\log(B(t))$ initially decreases faster than the line $L=-\frac{2A_0}{5\nu e^2}t$, followed by less rapid decay. The point of intersection at $t=120$ was a choice made to most clearly show the change in decay rate of the higher modes once $t>\frac{1}{\nu}$. An analogous result will be shown for $\delta\neq 1$ in  \S\ref{AsymmetricTorus} in Lemma \ref{globaldecay} and Proposition \ref{Brapiddecay}.

\begin{figure}[!htb]
	\centering
	\includegraphics[width=10cm, height=6cm]{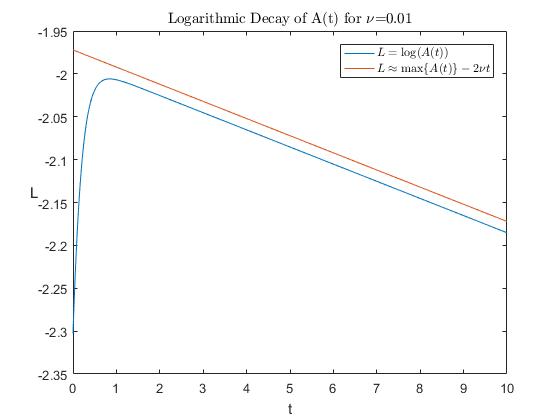}
	\caption{$A(t)=|\omega_1|^2+|\omega_3|^2$ decays at the background decay rate.}
	\label{figlogA}
\end{figure}

\begin{figure}[!htb]
	\centering
	\includegraphics[width=10cm, height=6cm]{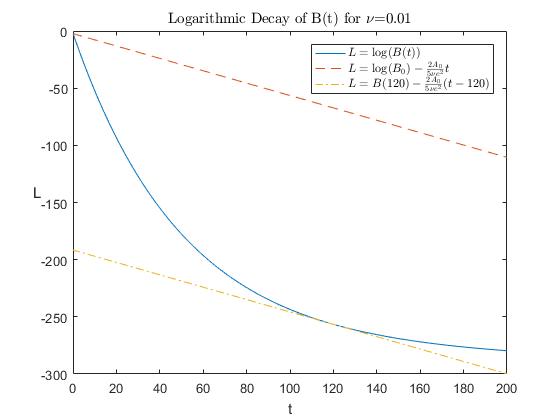}
	\caption{$B(t)=|\omega_5|^2+|\omega_7|^2$ has an initial period of rapid decay.}
	\label{figlogB}
\end{figure}

\begin{Remark} Interestingly, one can check that if the general forms of $G_{(2,0)}$ and $G_{(0,2)}$ given by \eqref{E:G} were used in \eqref{E:symmetric}, $\dot{A}$ and $\dot{B}$ as defined in the statement of Lemma \ref{decaydelta=1} would still satisfy \eqref{E:AB}.  Therefore, the dynamics of $A$ and $B$ are exactly the same regardless of the choice of coefficients for $G_{(2,0)}$ and $G_{(0,2)}$.
\end{Remark}

Next, to precisely describe the behavior exhibited in Figure \ref{figR}, we define the following variables (see also Remark \ref{R:cov}):
\begin{equation}\label{E:changevars}
R = \frac{|\omega_1|^2}{|\omega_3|^2}, \quad A = |\omega_1|^2 + |\omega_3|^2, \quad w = |\omega_5|^2, \quad
z = |\omega_7|^2, \quad P = \frac{\omega_1\bar{\omega}_3\bar{\omega}_7}{|\omega_3|^2}, \quad  Q = \frac{\bar{\omega}_1\bar{\omega}_3\omega_5}{|\omega_3|^2}.
\end{equation}
Equation \eqref{E:symmetric} then implies that
\begin{equation}
    \begin{aligned}
        \dot{R} &= (1+R)(P_{re}-Q_{re}) \\
        \dot{A} &= -2\nu A + \frac{3}{20\nu}A(w+z)  \\
        \dot{w} &= -4\nu w - \frac{2}{5\nu}wA  \\
        \dot{z} &= -4\nu z -\frac{2}{5\nu}zA  \label{E:PQsystem}\\
        \dot{P} &= -2\nu P +\frac{z}{2}(1-R)-\frac{1}{5\nu}PA+P(P_{re}-Q_{re})+\frac{P}{2}(\bar{Q}-\frac{Q}{R}) \\
        \dot{Q} &= -2\nu Q +\frac{w}{2}(R-1)-\frac{1}{5\nu}QA+Q(P_{re}-Q_{re})+\frac{Q}{2}(\frac{P}{R}-\bar{P})
    \end{aligned}
\end{equation}
where $P_{re}=\Re(P)$, $P_{im}=\Im(P)$, $Q_{re}=\Re(Q)$, and $Q_{im}=\Im(Q)$.

\begin{Theorem} \label{Symmetrictheorem}
There exists a family of local stable manifolds, $\mathcal{M}_r=\{R=f(A,w,z,P_{re},P_{im},Q_{re},Q_{im};r)\}$ for $r\geq0$, corresponding to each of the fixed points of the system \eqref{E:PQsystem}. Any initial condition on $\mathcal{M}_r$ will converge at an exponential rate to the fixed point $\{ R=r, A=0, w=0, z=0, P=0, Q=0\}$ as $t\to\infty$. This family, up to and including quadradic terms, is given by
\begin{equation}
    \begin{aligned}
f &= r+\frac{r^2-1}{16\nu^2}w +\frac{r^2-1}{16\nu^2}z - \frac{r+1}{2\nu}P_{re}+\frac{r+1}{2\nu}Q_{re}-\frac{r^2-1}{160\nu^4}Aw-\frac{r^2-1}{160\nu^4}Az +\frac{r+1}{40\nu^3}AP_{re}-\frac{r+1}{40\nu^3}AQ_{re}  \\
&  \qquad +\frac{r^2-1}{768\nu^4r}(7r^2+2r+1)wz -\frac{r+1}{96\nu^3r}(4r^2-r+1)wP_{re} + \frac{r+1}{96\nu^3}(3r+1)wQ_{re}+\frac{r^2-1}{768\nu^4}(3r+2)w^2  \\
&  \qquad -\frac{r+1}{96\nu^3}(3r+1)zP_{re}+\frac{r+1}{96\nu^3r}(4r^2-r+1)zQ_{re} +\frac{r^2-1}{768\nu^4}(3r+2)z^2 -\frac{r^2-1}{8\nu^2r}P_{re}Q_{re}-\frac{(r+1)^2}{8\nu^2r}P_{im}Q_{im}.  \nonumber
    \end{aligned}
\end{equation}
\end{Theorem}

\begin{proof}
The existence of this family of stable manifolds is given by Theorem 4.1 in \cite{Chicone}.  Before applying this theorem, it must be verified that the nonlinearity of \eqref{E:PQsystem}, as well as its partial derivatives, vanish at each fixed point. To see that no singularities exist, note that the domain of \eqref{E:PQsystem} is determined by the way the variables are defined in \eqref{E:changevars}. It is given by the set

\[\left\{\left(R=\frac{|\omega_1|^2}{|\omega_3|^2},A=|\omega_1|^2 + |\omega_3|^2,w= |\omega_5|^2,z= |\omega_7|^2,P= \frac{\omega_1\bar{\omega}_3\bar{\omega}_7}{|\omega_3|^2},Q=\frac{\bar{\omega}_1\bar{\omega}_3\omega_5}{|\omega_3|^2}\right)\in\mathbb{R}^4\times \mathbb{C}^2:|\omega_3|^2>0 \right\}.\]

The vector field in \eqref{E:PQsystem} has no singular points in this domain since,
\[\left|\frac{PQ}{R}\right|=\frac{|\omega_1|^2||\omega_3|^2|\omega_5||\omega_7|}{|\omega_3|^4}\frac{|\omega_3|^2}{|\omega_1|^2}=|\omega_5||\omega_7|.
\]
Therefore, at any fixed point of the form $\vec{r}=(r,0,0,0,0,0,0,0)$, it can be checked that the nonlinearity and all of its partial derivatives vanish. Thus we obtain the existence of a local invariant (stable) manifold $\mathcal{M}_r$, associated to the flow toward the fixed point at $\vec{r}$. The rest of the proof is a long, but standard, computation of the family $\mathcal{M}_r$ and will be left to the appendix.
\end{proof}

The implication Theorem \ref{Symmetrictheorem} has on the originial system \eqref{E:symmetric} for $\omega_1,$ $\omega_3,$ $\omega_5,$ and $\omega_7$ is that for a typical initial condition, system \eqref{E:symmetric} will not evolve to an $x$-bar ($r = \infty$) or $y$-bar ($r = 0$) state. For most $r \in (0, \infty)$, the asymptotic state will look like a dipole. We note, however, that for $r \ll 1$ or $r \gg 1$ the asymptotic state may appear to be more like a bar state than a dipole, even though it is not a pure bar state. This is again consistent with the numerical results of \cite{BouchetSimonnet09,Yin}.

\begin{Remark} \label{R:cov}
Theorem \ref{Symmetrictheorem} establishes the existence of locally invariant manifolds only for initial conditions to \eqref{E:symmetric} with $|\omega_3|\neq 0$. The same result can be proven for initial conditions to \eqref{E:symmetric} with $|\omega_1|\neq 0$ by taking instead as the change of variables $\mathcal{R}=\frac{|\omega_3|^2}{|\omega_1|^2}$, $\mathcal{P}= \frac{\omega_1\bar{\omega}_3\bar{\omega}_7}{|\omega_1|^2}$, and  $\mathcal{Q} = \frac{\bar{\omega}_1\bar{\omega}_3\omega_5}{|\omega_1|^2}$ with $A$, $w$, and $z$ unchanged. The set $\{ \omega_1 = \omega_3 = 0\}$ is invariant for \eqref{E:symmetric}, so this case is not relevant for the study of the selection of the bars versus dipoles.
\end{Remark}

\begin{Remark}
The invariant manifolds $\mathcal{M}_r$ in Theorem \ref{Symmetrictheorem} are defined locally near each fixed point $\vec{r}$. However, Lemma \ref{lem:sym-decay} implies that, for any initial condition, the solution will eventually enter a neighborhood of this one-dimensional manifold of fixed points, and hence its dynamics will eventually be governed by the theorem.
\end{Remark}

\begin{Remark} If we instead use the more general form of $\hat{\omega}(2,0)$ and $\hat{\omega}(0,2)$ on the center manifold given by \eqref{E:G}, one can check that, although the explicit form of $\mathcal{M}_r$ would change, qualitatively the same structure will prevail. There remains a line of fixed points as the one-dimensional center manifold and a family of stable manifolds consisting of the initial conditions that evolve towards each of these fixed points.
\end{Remark}

%%%%%%%%%%%%%%%%%%%%%%%%%%%%%%%%%%%%

\section{Asymmetric Torus}\label{AsymmetricTorus}

In this section we study \eqref{E:deltanot1} with $\delta \neq 1$. In particular, we show in Proposition \ref{Brapiddecay} that again the high modes decay much more rapidly than the low modes, and we show in Theorem \ref{Asymmetrictheorem} that $R(t) \to 0$ as $t \to \infty$ for $\delta < 1$, thus indicating convergence to a $y$-bar state. A similar convergence result to an $x$-bar state exists when $\delta > 1$. We begin by rewriting \eqref{E:deltanot1} for convenience:
\begin{equation}
    \begin{aligned}
        \dot \omega_1 &= - \frac{\nu}{\delta^2} \omega_1 + \frac{1}{\delta(1+\delta^2)}[\omega_3\omega_7 - \bar{\omega}_3 \omega_5] + \frac{3\delta^6}{2\nu(4+\delta^2)(1+\delta^2)^2}\omega_1(|\omega_5|^2 + |\omega_7|^2)  \\
        \dot \omega_3 &= - \nu \omega_3 + \frac{\delta^3}{(1+\delta^2)}[\bar{\omega}_1\omega_5 - \omega_1 \bar{\omega}_7] + \frac{3\delta^2}{2\nu(1+4\delta^2)(1+\delta^2)^2}\omega_3(|\omega_5|^2 + |\omega_7|^2) \label{E:asymmetric}  \\
        \dot \omega_5 &= - \nu \frac{1+\delta^2}{\delta^2} \omega_5 -\frac{\delta^2-1}{\delta}\omega_1 \omega_3   -\frac{\delta^6(3+\delta^2)}{2\nu(4+\delta^2)(1+\delta^2)} \omega_5 |\omega_1|^2-\frac{1+3\delta^2}{2\nu\delta^2(1+4\delta^2)(1+\delta^2)}\omega_5 |\omega_3|^2    \\
        \dot \omega_7 &= - \nu \frac{1+\delta^2}{\delta^2} \omega_7 +\frac{\delta^2-1}{\delta}\omega_1 \bar{\omega}_3 -\frac{\delta^6(3+\delta^2)}{2\nu(4+\delta^2)(1+\delta^2)} \omega_7 |\omega_1|^2 - \frac{1+3\delta^2}{2\nu\delta^2(1+4\delta^2)(1+\delta^2)}\omega_7 |\omega_3|^2.
    \end{aligned}
\end{equation}
For the remainder of this section, we will assume that $\delta$ is fixed with $\delta\in (\sqrt{\frac{2}{3}},\sqrt{\frac{3}{2}})$, $\delta \neq 1$. It will occasionally be useful to fix a sufficiently small value of $\eta > 0$ and require that $|\delta^2-1|<\eta$. For such a fixed $\delta$ and $\eta$, some of our results will hold for all $\nu>0$ sufficiently small.

First, we obtain a global decay rate via the energy
\[
E(t):= \frac{1}{2}(|\omega_1(t)|^2+|\omega_3(t)|^2+|\omega_5(t)|^2+|\omega_7(t)|^2).
\]
 \begin{Lemma}\label{globaldecay}
 For all $t \geq 0$, $E(t)\leq E(0)e^{-2K_1\nu t}$, with $K_1=\min({1,1/\delta^2})$.
\end{Lemma}
\begin{proof}
Using (\ref{E:asymmetric}) we see that
\begin{eqnarray}
\dot{E}(t) &=& -\nu\left(\frac{1}{\delta^2}|\omega_1|^2+|\omega_3|^2+\frac{1+\delta^2}{\delta^2}(|\omega_5|^2+|\omega_7|^2)\right) - \frac{\delta^8}{2\nu(1+\delta^2)^2}|\omega_1|^2(|\omega_5|^2+|\omega_7|^2) \nonumber \\
& & - \frac{1}{2\nu\delta^2(1+\delta^2)^2}|\omega_3|^2(|\omega_5|^2+|\omega_7|^2) \nonumber \\
&\leq& -2\nu\min(1,1/\delta^2)E, \nonumber
\end{eqnarray}
which proves the result.
\end{proof}

To obtain a faster decay rate for the higher modes, we again define $A=|\omega_1|^2+|\omega_3|^2$ and $B=|\omega_5|^2+|\omega_7|^2$, and denote the real and imaginary parts of $\omega_i$ as $\omega_i^{re}$ and $\omega_i^{im}$ respectively for $i=1,3,5,7.$ We find using \eqref{E:asymmetric} that
\begin{eqnarray}
\dot A &=& -2\frac{\nu}{\delta^2}|\omega_1|^2 - 2\nu|\omega_3|^2 + \frac{3\delta^2}{\nu(1+\delta^2)^2}\left[\frac{\delta^4}{4+\delta^2}|\omega_1|^2+\frac{1}{1+4\delta^2}|\omega_3|^2\right]B \nonumber \\
& & +2\frac{\delta^2-1}{\delta}[\omega_1^{re}\omega_3^{re}(\omega_7^{re}-\omega_5^{re})+\omega_1^{im}\omega_3^{re}(\omega_7^{im}-\omega_5^{im})-\omega_1^{re}\omega_3^{im}(\omega_5^{im}+\omega_7^{im})+\omega_1^{im}\omega_3^{im}(\omega_5^{re}+\omega_7^{re})] \nonumber \\
\dot B &=& -2\nu\frac{1+\delta^2}{\delta^2}B -\frac{1}{\nu}\left[\frac{\delta^6(3+\delta^2)}{(4+\delta^2)(1+\delta^2)}|\omega_1|^2+\frac{1+3\delta^2}{\delta^2(1+4\delta^2)(1+\delta^2)}|\omega_3|^2\right] B \nonumber \\
& & + 2\frac{\delta^2-1}{\delta}[\omega_1^{re}\omega_3^{re}(\omega_5^{re}-\omega_7^{re})+\omega_1^{im}\omega_3^{re}(\omega_5^{im}-\omega_7^{im})+\omega_1^{re}\omega_3^{im}(\omega_5^{im}+\omega_7^{im})-\omega_1^{im}\omega_3^{im}(\omega^r_5+\omega^r_7)]\nonumber
\end{eqnarray}.

\begin{Proposition} \label{Brapiddecay}Let $A_0 = A(0)$, $B_0 = B(0)$, and $K_1 = \min(1, 1/\delta^2)$. For any fixed $\delta$ and $\eta$ with {$0<|\delta^2-1|<\eta$},  there exists a $\nu^*>0$ such that for all $0<\nu<\nu^*$ there exists a time $t_* = \mathcal{O}(\nu|\log (\nu\eta)|)$ and positive constants $M_0$, $M_1$ and $K_2$, independent of $\nu$, such that
	\[
	B(t) \leq \begin{cases} B_0 e^{-\frac{M_0}{\nu}t} &\mbox{ for } 0 \leq t \leq t_* \\
	M_1 \eta^2\nu^2 e^{-2\nu K_1 t} &\mbox{ for } t_* \leq t.
	\end{cases}
	\]
	Moreover, for all $0 \leq t \leq \min(1/\nu, 1/\eta)$, $A(t) \geq A_0 e^{-K_2}$, and for all $t \geq 0$, $A(t) + B(t) \leq (A_0 + B_0) e^{-2\nu K_1 t}$.
\end{Proposition}

\begin{proof}
By Lemma \ref{globaldecay}, we have $A(t) + B(t) \leq (A_0 +B_0) e^{-2\nu K_1 t}$. Next, notice that
\[
|\omega_5 - \omega_7| \leq |\omega_5| + |\omega_7| \leq \sqrt{2B} \leq \sqrt{2(A+B)} \leq \sqrt{2(A_0 +B_0)} e^{-\nu K_1 t}.
\]
Using the fact that $|\omega_1\omega_3| \leq (|\omega_1|^2 + |\omega_3|^2)/2=A/2$, we have
\[
\dot A \geq -2\nu \mbox{max}\left(1, \frac{1}{\delta^2}\right)A - \frac{\eta}{\delta^2}A|\omega_5-\omega_7| \geq -2\nu \mbox{max}\left(1, \frac{1}{\delta^2}\right)A - 4\sqrt{2}\frac{\eta}{\delta^2}A\sqrt{(A_0+B_0)}.
\]
Thus, for all $t \geq 0$,
\[
A(t) \geq A_0 \mbox{exp}\left[ -\left(2\nu \mbox{max}\left(1, \frac{1}{\delta^2}\right) + 4\sqrt{2}\frac{\eta}{\delta^2}\sqrt{A_0+B_0} \right) t\right].
\]
If we now let $t_1 = \min(1/\nu, 1/\eta)$, then  $A(t) \geq A_0 e^{-K_2}$ for all $0 \leq t \leq t_1$, where
\[
K_2 = \left(2\mbox{max}\left(1, \frac{1}{\delta^2}\right) + 6\sqrt{2}\sqrt{(A_0+B_0)} \right) \geq \left(2\nu \mbox{max}\left(1, \frac{1}{\delta^2}\right) + 4\sqrt{2}\frac{\eta}{\delta^2}\sqrt{(A_0+B_0)} \right) t_1.
\]

Here we used the assumption that $\delta\geq\sqrt{\frac{2}{3}}$. To obtain the estimates for $B$, notice that the above equation for $\dot B$ implies
\[
\dot B \leq -2\frac{\nu}{\delta^2}(1+\delta^2)B + 4\sqrt{2}\frac{\eta}{\delta}A\sqrt{B} - \frac{1}{\nu} D_0 AB
\]
where
\[
D_0 = \min\left(\frac{(1+3\delta^2)}{\delta^2(1+\delta^2)(1+4\delta^2)}, \frac{\delta^6(3+\delta^2)}{(1+\delta^2)(4+\delta^2)}\right).
\]
Since $\delta\in(\sqrt{\frac{2}{3}},\sqrt{\frac{3}{2}})$, $D_0$ is bounded away from 0. Suppose first that
\[
B(t) \geq B^* := \frac{128 \nu^2 \eta^2}{\delta^2D_0^2}.
\]
(If the initial condition $B_0 < B_*$, then it is already asymptotically small and we will consider this case afterwards.) Then we find that
\begin{equation}\label{E:Bineq}
\dot B \leq -2\frac{\nu}{\delta^2}(1+\delta^2)B - \frac{1}{2\nu} D_0 AB \leq -\frac{1}{2\nu} D_0 A_0 e^{-K_2} B.
\end{equation}
Thus, setting $M_0 = (D_0 A_0)/(2e^{K_2})$ we find that $B(t) \leq B_0 e^{-\frac{M_0}{\nu}t}$, at least until $B(t) = B^*$. The latter occurs at a time no larger than $t_*$, where $t_*$ is defined via
\[
B_0 e^{-\frac{M_0}{\nu}t_*} = B_* \qquad \Rightarrow \qquad t_* = -\frac{\nu}{M_0} \log \left( \frac{128\nu^2\eta^2}{B_0 \delta^2 D_0^2}\right) = \mathcal{O}(\nu |\log(\nu\eta)|) \mbox{ as } \nu\to0.
\]	
Finally, consider times for which $B(t) \leq B_*$. By Lemma \ref{globaldecay}, $B(t)\leq B_0 e^{-2\nu K_1 t}$ for all time. Hence, if we restart the system at time $t=t_*$, when $B=\mathcal{O}(\eta^2\nu^2)$, then $B(t)$ decays at the background rate for $t\geq t_*$, i.e. $B(t)\leq M_1\eta^2\nu^2 e^{-2\nu K_1 t}$. This completes the proof.
\end{proof}

\begin{Remark} We comment on the relationship between Proposition 4.2 and Lemma 3.1. Lemma 3.1 is a result in the case $\delta = 1$. In Proposition 4.2, one must begin by fixing $\delta$ and $\eta$ so that $0 < |\delta^2 -1 | < \eta$. Although one can take $\delta$ arbitrarily close to $1$, but fixed, and hence $\eta$ arbitrarily close to zero, but fixed, one cannot take the limits $\delta \to 1$ and $\eta \to 0$ and obtain something meaningful from Proposition 4.2. Moreover, although $t_* \to 0$ as $\nu \to 0$, one must also be careful with this limit, because the bound on $|B(t)|$ also goes to zero in this limit: if we were to choose $t = t_*/2$ in this bound, we would have
 \[
 |B(t_*/2)| \leq B_0 e^{-\mathcal{O}(|\log(\nu\eta)|)},
 \] 
 which goes to zero as $\nu \to 0$. Thus, one really must keep $\delta, \eta$, and $\nu$ fixed when considering Proposition 4.2.

 We further note that the results of Proposition 4.2 do not contradict Lemma 3.1. 
%In Proposition \ref{Brapiddecay} one first chooses $\delta\neq 1$ and $\eta>0$ and then chooses $\nu^{*}$ and $t_{*}$. The statement of Proposition \ref{Brapiddecay} does not directly translate to that of Lemma \ref{lem:sym-decay} which is for $\eta=0$. The reason for that is that the assumption that $\eta>0$ is a-priori given and is fixed was used in the proof of Proposition \ref{Brapiddecay}. However, the results of Proposition \ref{Brapiddecay}, which are for $\delta\neq 1$ and $\eta>0$ a-priori chosen and fixed, do not contradict Lemma \ref{lem:sym-decay}, which is for $\delta=1$. 
Both Lemma \ref{lem:sym-decay} and Proposition \ref{Brapiddecay} say that for an initial period of time, $B(t)$ is bounded above by a term of the order $O(e^{-\frac{\Lambda}{\nu}t})$, albeit with slightly different constants $\Lambda>0$ in each one of the two cases. After that initial period of time, both results say that $B(t)$ remains small, and to be precise in Lemma \ref{lem:sym-decay} we showed that for $t$ large enough $B(t)$ is bounded above by $e^{-\frac{2A_0}{5\nu^2 e^2}}$, whereas in Proposition \ref{Brapiddecay} we can prove an upper bound of the order $\eta^2\nu^2 e^{-2\nu K_1 t}$.  In the case $\delta\neq 1$, we can only guarantee the upper bounds that appear in Proposition \ref{Brapiddecay}, which however are sufficient for our purposes, as Theorem \ref{Asymmetrictheorem} below demonstrates.
\end{Remark}

Now consider again the variable, $R=\frac{|\omega_1|^2}{|\omega_3|^2}$. Suppose $\delta<1$ and set $\gamma = 2\nu(\frac{1}{\delta^2} -1) \leq C\nu \eta$. We wish to show that $R(t) \sim \mathcal{O}(e^{-\gamma t})$ for all $t \geq 0$, as long as $R(0) =: R_0$ is not too big. Using \eqref{E:asymmetric}, we find
\begin{align}
\dot{R} &= -\gamma R + \frac{3\delta^2}{\nu(1+\delta^2)^2}\left(\frac{\delta^4}{4+\delta^2}-\frac{1}{1+4\delta^2}\right)RB \nonumber \\
& \quad +\frac{2}{\delta(1+\delta^2)}\frac{1}{|\omega_3|^2}\left[\omega_1^{re}\omega_3^{re}(\omega_7^{re}-\omega_5^{re})-\omega_1^{re}\omega_3^{im}(\omega_5^{im}+\omega_7^{re})+\omega_1^{im}\omega_3^{re}(\omega_7^{im}-\omega_5^{im})+\omega_1^{im}\omega_3^{im}(\omega_5^{re}+\omega_7^{re})\right]  \nonumber \\
& \quad +\frac{2\delta^3}{1+\delta^2}R\frac{1}{|\omega_3|^2}\left[\omega_1^{re}\omega_3^{re}(\omega_7^{re}-\omega_5^{re})-\omega_1^{re}\omega_3^{im}(\omega_5^{im}+\omega_7^{re})+\omega_1^{im}\omega_3^{re}(\omega_7^{im}-\omega_5^{im})+\omega_1^{im}\omega_3^{im}(\omega_5^{re}+\omega_7^{re})\right]. \nonumber
\end{align}

\begin{Theorem}\label{Asymmetrictheorem}
For any fixed $\delta < 1$ , sufficiently close to 1, and fixed $\eta$ sufficiently small, with $|\delta^2-1|<\eta$, there exist $R^*, \nu^*>0$ such that, for all $0<\nu\leq\nu^*$ and $R_0 \leq R^*$, $R(t) \leq M_2 e^{-\gamma t}$ for some $M_2 > 0$ and all $t \geq 0$. $R_*$ is $\mathcal{O}(\min(\eta^{-1},\nu^{-1}))$ while $M_2$ is independent of $\nu$ and $\eta$.
\end{Theorem}

\begin{proof}
For ease of notation, define
\begin{gather*}
\beta_1 = \frac{3\delta^2}{(\delta^2-1)(1+\delta^2)^2}\left( \frac{\delta^4}{4+\delta^2} - \frac{1}{1+4\delta^2}\right) = \frac{3\delta^2}{(1+\delta^2)^2}\left( \frac{\delta^4 + 2 \delta^2 + 1}{(4+\delta^2)(1+4\delta^2)}\right), \\
\beta_2 = \frac{2\delta^3}{1+\delta^2}, \qquad \beta_3 = \frac{2}{\delta(1+\delta^2)}.
\end{gather*}
With this notation,
\begin{align}
\dot{R} &\leq -\gamma R + \frac{\eta}{\nu}\beta_1 RB + 2\sqrt{2}\beta_3\frac{1}{|\omega_3|^2}A\sqrt{B} + 2\sqrt{2}\beta_2 \frac{1}{|\omega_3|^2}RA\sqrt{B} \nonumber \\
&= -\gamma R +\frac{\eta}{\nu}\beta_1 RB + 2\sqrt{2}\beta_3 (1+R)\sqrt{B}+2\sqrt{2}\beta_2(1+R)R\sqrt{B} \nonumber \\
&= -\gamma R + \left(\frac{\eta}{\nu}\beta_1 B +2\sqrt{2}(\beta_2+\beta_3)\sqrt{B}\right)R + 2\sqrt{2}\beta_2 R^2\sqrt{B}+2\sqrt{2}\beta_3 \sqrt{B}.  \nonumber
\end{align}
Therefore, we have the bound
\begin{eqnarray*}
&& R(t) \leq e^{-\gamma t}R_0 \\
&& \quad \quad + \int_0^t e^{-\gamma(t-s)} \left[ \left(\frac{\eta}{\nu} \beta_1 B(s) + 2\sqrt{2}(\beta_2+\beta_3)\sqrt{B(s)}\right)R(s) +2\sqrt{2}\beta_2 R^2(s)\sqrt{B(s)}+2\sqrt{2}\beta_3 \sqrt{B(s)}  \right] \rmd s.
\end{eqnarray*}
Define
\[
|||R||| := \sup_{0 \leq t \leq T} e^{\gamma t }R(t),
\]
where $T$ is defined to the  largest time such that $e^{\gamma t} R(t) \leq R_*$, and $R_*$ will be determined below. If $|||R|||$ is finite, then $R(t)$ decays like $e^{-\gamma t}$.  Multiplying the integral inequality above by the exponential weight $e^{\gamma t}$, we find that
\begin{eqnarray*}
&& |||R||| \leq R_0 \\
&& \quad + \sup_{0 \leq t \leq T} \int_0^t e^{\gamma s} \left[ \left(\frac{\eta}{\nu} \beta_1 B(s) + 2\sqrt{2}(\beta_2+\beta_3)\sqrt{B(s)}\right)R(s) +2\sqrt{2}\beta_2 R^2(s)\sqrt{B(s)}+2\sqrt{2}\beta_3 \sqrt{B(s)} \right] \rmd s \\
&& \quad \leq R_0 + \sup_{0 \leq t \leq T}(I + II + III).
\end{eqnarray*}
where
\begin{eqnarray*}
I &=& |||R||| \sup_{0 \leq t \leq T} \int_0^t \left(\frac{\eta}{\nu} \beta_1 B(s)  + 2\sqrt{2}(\beta_2+\beta_3)\sqrt{B(s)}\right) \rmd s \\
II &=& |||R|||^2 \sup_{0 \leq t \leq T} \int_0^t 2\sqrt{2}\beta_2 e^{-\gamma s}\sqrt{B(s)} \rmd s \\
III &=&\sup_{0 \leq t \leq T} \int_0^t 2\sqrt{2}\beta_3 e^{\gamma s}\sqrt{B(s)} \rmd s.
\end{eqnarray*}		
We now estimate the three terms above using Proposition \ref{Brapiddecay}, by splitting the time interval into two pieces: $0 \leq t \leq t_*$ and $t_* \leq t$. So Proposition \ref{Brapiddecay} gives for term I,
\begin{eqnarray*}
I &\leq& |||R|||  \int_{0}^{t_*} \left(\frac{\eta }{\nu} \beta_1 B_0e^{-\frac{M_0}{\nu}s} + 2\sqrt{2}(\beta_2 + \beta_3)\sqrt{B_0}e^{-\frac{M_0}{2\nu}s}\right) \rmd s \\
&& \quad + |||R||| \int_{t_*}^{t} \left(M_1 \eta^3\nu e^{-2\nu K_1 s} + \sqrt{M_1}2\sqrt{2}(\beta_2 + \beta_3)\eta\nu e^{-\nu K_1s}\right)  \rmd s  \\
&\leq& |||R||| \left[ \frac{ \eta \beta_1 B_0}{M_0}\left( 1 - e^{-\frac{M_0}{\nu}t_*}\right) + \frac{4\nu\sqrt{2}(\beta_2 + \beta_3)\sqrt{B_0}}{M_0}\left(1-e^{-\frac{M_0}{2\nu}t_*}\right)\right] \\
&& \quad  + |||R||| \left[\frac{M_1 \eta^3}{2K_1}\left(e^{-2\nu K_1 t_*}-e^{-2\nu K_1 t}\right)  +\frac{\eta\sqrt{M_1}2\sqrt{2}(\beta_2 + \beta_3)}{K_1}\left(e^{-\nu K_1 t_*}-e^{-\nu K_1 t}\right)\right] \\
&\leq& C_1( \eta + \eta^3 +  \nu)|||R|||,
\end{eqnarray*}
for some constant $C_1$ that is independent of $\nu$ and $\eta$. Similarly, for term II we have
\begin{align*}
II &\leq 2\sqrt{2}\beta_2 |||R|||^2 \left[ \int_0^{t_*} \sqrt{B_0}e^{-(\frac{M_0}{2\nu} + \gamma) s} \rmd s + \int_{t_*}^{t} \sqrt{M_1} \eta \nu e^{-( \nu K_1 + \gamma)s} \rmd s \right] \\
&\leq 2\sqrt{2}\beta_2|||R|||^2\left[ \frac{2\nu\sqrt{B_0}}{M_0 + 2\nu \gamma}\left( 1 - e^{-(\frac{M_0}{2\nu} + \gamma) t_*}\right) + \frac{\sqrt{M_1}\eta \nu}{\nu K_1+\gamma}\left( e^{-(\nu K_1 + \gamma) t_*} - e^{-(\nu K_1  + \gamma) t} \right) \right] \\
&\leq C_2|||R|||^2(\nu+\eta),
\end{align*}
for some constant $C_2$ that is independent of $\nu$ and $\eta$. Finally, for term III we have
\begin{align*}
III &\leq 2\sqrt{2}\beta_3 \left[ \int_0^{t_*} \sqrt{B_0}e^{-(\frac{M_0}{2\nu}s - \gamma) s} \rmd s + \int_{t_*}^{t} \sqrt{M_1} \eta \nu e^{-(\nu K_1-\gamma) s} \rmd s \right] \\
&\leq 2\sqrt{2}\beta_3 \left[ \frac{2\nu\sqrt{B_0}}{M_0 - 2\nu \gamma}\left( 1 - e^{-(\frac{M_0}{2\nu} - \gamma) t_*}\right) + \frac{\sqrt{M_1}\eta \nu}{\nu K_1-2\gamma}\left( e^{-(\nu K_1 - \gamma) t_*} - e^{-(\nu K_1 - \gamma) t} \right) \right]\nonumber  \\
&\leq C_3(\nu  + \eta ),
\end{align*}
for some constant $C_3$ that is independent of $\nu$ and $\eta$. Combining the above estimates, we find that
\[
|||R||| \leq  R_0 + \mbox{max}(\eta,\nu) C_3  + \mbox{max}(\eta,\nu) C_1|||R||| + \mbox{max}(\eta,\nu) C_2|||R|||^2 .
\]
We can rewrite this as
\[
|||R|||\left( 1 - \mbox{max}(\eta,\nu) C_1 - \mbox{max}(\eta,\nu) C_2 |||R|||\right) \leq R_0 + \mbox{max}(\eta,\nu) C_3.
\]
Now, choose $R_* = 1/(2\mbox{max}(\eta,\nu) C_2)$. We then find
\[
|||R||| \leq \frac{R_0 + \mbox{max}(\eta,\nu) C_0}{1 - \mbox{max}(\eta,\nu) C_1 - \mbox{max}(\eta,\nu) C_2 |||R|||} \leq 4(R_0 + \mbox{max}(\eta,\nu) C_0),
\]
as long as $\nu^*$ and $\eta$ are such that $\max(\eta,\nu^*) \leq \frac{1}{4C_1}$. Using the above definition of $R^*$, the above right hand side will be less than $R^*$ if $R_0 \leq R^*/8$ and $\max(\eta,\nu^*) \leq 1/\sqrt{16C_0C_2}$. In this case,
\[
\sup_{0 \leq t \leq T} e^{\gamma t }|R(t)| = |||R||| \leq R_*.
\]
Because the bound on the right hand side is independent of $T$, the bound must hold for all time. Hence, there must exist an $M_2 > 0$ such that $R(t) \leq M_2 e^{-\gamma t}$ for all $t \geq 0$.
\end{proof}

We emphasize again the importance of the transient period of rapid decay present in the dynamics of the higher Fourier modes established in Proposition \ref{Brapiddecay}. Without their rapid decay to a small enough order during this initial time period, the estimates in the proof of Theorem \ref{Asymmetrictheorem}, particularly for term I, would not have gone through. Moreover, we see that given a small, fixed distance of $\delta$ from 1, a sufficiently small value for the viscosity can be selected to separate the decay rates into the two regimes established in Proposition \ref{Brapiddecay} and used in the proof of Theorem \ref{Asymmetrictheorem}, thus driving the system toward a bar state.

Let us now define $U(t)=R(t)^{-1}$. We see that $U(t)$ must satisfy
\begin{align*}
\dot{U} &= \gamma U - \frac{3\delta^2}{\nu(1+\delta^2)^2}\left(\frac{\delta^4}{4+\delta^2}-\frac{1}{1+4\delta^2}\right)UB \nonumber \\
& \quad -\frac{2}{\delta(1+\delta^2)}U\frac{1}{|\omega_1|^2}\left[\omega_1^{re}\omega_3^{re}(\omega_7^{re}-\omega_5^{re})-\omega_1^{re}\omega_3^{im}(\omega_5^{im}+\omega_7^{re})+\omega_1^{im}\omega_3^{re}(\omega_7^{im}-\omega_5^{im})+\omega_1^{im}\omega_3^{im}(\omega_5^{re}+\omega_7^{re})\right] \nonumber \\
& \quad -\frac{2\delta^3}{1+\delta^2}\frac{1}{|\omega_1|^2}\left[\omega_1^{re}\omega_3^{re}(\omega_7^{re}-\omega_5^{re})-\omega_1^{re}\omega_3^{im}(\omega_5^{im}+\omega_7^{re})+\omega_1^{im}\omega_3^{re}(\omega_7^{im}-\omega_5^{im})+\omega_1^{im}\omega_3^{im}(\omega_5^{re}+\omega_7^{re})\right] \nonumber
\end{align*}

The dynamics of $U$ for $\delta>1$ are analogous to the dynamics of R when $\delta < 1$. With similar estimates to those in the proof of Theorem \ref{Asymmetrictheorem}, one can show that $U(t)\to0$ as $t\to\infty$, which indicates convergence to an $x$-bar state.

%%%%%%%%%%%%%%%%%%%%%%%%%%%%%%%%%%%%%%%%%%%%%%%%%

\section{Perturbation Analysis}\label{GeomSingPert}
The purpose of this section is to provide further evidence of a selection mechanism through an alternate method. The system's domain will be viewed as a perturbation of the symmetric torus. Using the approximations computed in this section, we will confirm the results of \S \ref{SymmetricTorus} and \S\ref{AsymmetricTorus}. To ultimately view \eqref{E:deltanot1} as a perturbed system, we define the perturbation parameter $\epsilon$ via $\delta=1+\epsilon_0\epsilon$, with $\epsilon_0=\pm1$. Note that the sign of $\epsilon_0$ determines whether $\delta$ is greater or less than one. We begin by scaling $\nu$ as an appropriate power of $\epsilon$, which effectively relates $\nu$ and $\delta$. Subsequently, asymptotic expansions in $\epsilon$ are computed that are connected with the observed multiple time scales in the evolution of the vorticity. The properties of these expansions agree with the results of \S\ref{SymmetricTorus} and \S\ref{AsymmetricTorus}. Moreover, we are again able to observe the evolution to the appropriate bar state depending on if $\epsilon_0 = \pm 1$. These expansions also have the property, at least among the $\mathcal{O}(1)$ and $\mathcal{O}(\epsilon)$ terms, that the emergence of a bar state will become faster as $|\delta - 1| = \epsilon$ increases. This could be related to the observation in \cite{BouchetSimonnet09} that the bar states do not dominate the stochastic evolution (based on the stochastic forcing used there) unless $|\delta - 1|$ is sufficiently large.

Motivated by geometric singular perturbation theory, we first scale \eqref{E:deltanot1} in a way that reveals a slow and a fast subsystem. We must choose appropriate values for the scaling parameters that accomplish two main objectives. First, to reveal a slow-fast system, we aim for the leading order terms in the scaled versions of $\dot{\omega}_1$ and $\dot{\omega}_3$ to be some order of magnitude in $\epsilon$ higher than those in the scaled versions of $\dot{\omega}_5$ and $\dot{\omega}_7$. Second, to ensure that the decay rates in the newly scaled system match those seen in the previous sections, we would like the leading order terms, once scaled, to match the terms in (\ref{E:leadingorder}) below.

\begin{equation}
    \begin{aligned}
    \dot{\omega}_{1} &=  -\frac{\nu}{\delta^2} \omega_{1} + \mbox{h.o.t.}\\
     \dot{\omega}_{3} &=  -\nu \omega_{3} + \mbox{h.o.t.} \label{E:leadingorder}\\
    \dot{\omega}_{5,7} &= -\frac{1}{2\nu(1+\delta^2)}\omega_{5,7}\left(\frac{\delta^6(3+\delta^2)}{(4+\delta^2)}  |\omega_1|^2 + \frac{1+3\delta^2}{\delta^2(1+4\delta^2)} |\omega_3|^2\right) + \mbox{h.o.t.}
    \end{aligned}
\end{equation}
If these correspond to the leading order terms in the scaled system, then the decay rates seen in the asymptotic expansions that are to be computed will match those observed in previous sections.

First scale the viscoscity and time by $\nu = \epsilon^\alpha\nu_0$ and $\tau=\epsilon^{\alpha}t$.  Then scale the Fourier modes by $\omega_1=\epsilon^\beta \Omega_1$, $\omega_3=\epsilon^\beta \Omega_3,$ $\omega_5=\epsilon^\phi \Omega_5$, and $\omega_7=\epsilon^\phi \Omega_7$. The acceptable range of values for the three scaling parameters $\alpha$, $\beta$, and $\phi$ that are relevant to this discussion will now be identified.  The clearest way to do so will be to define $\rho$ and $\sigma$ as

\begin{equation}
    \begin{aligned}
        \sigma &:= \beta - \alpha \\
        \rho &:= \phi - \alpha. \nonumber
    \end{aligned}
\end{equation}

With these initial scalings, the system \eqref{E:initialscale} below is obtained. Note here that the coefficient $\frac{\delta^2-1}{\delta}$, for $\delta=1+\epsilon_0\epsilon$ is $\mathcal{O}(\epsilon)$ as $\epsilon\to0$.
\begin{eqnarray}
        \frac{d}{d\tau} \Omega_1 &=& - \frac{\nu_0}{\delta^2} \Omega_1 + \epsilon^{\rho}\frac{1}{\delta(1+\delta^2)}[\Omega_3\Omega_7 - \bar{\Omega}_3 \Omega_5] + \epsilon^{2\rho}\frac{3\delta^6}{2\nu_0(4+\delta^2)(1+\delta^2)^2}\Omega_1(|\Omega_5|^2 + |\Omega_7|^2) \nonumber \\
        \frac{d}{d\tau} \Omega_3 &=& - \nu_0 \Omega_3 + \epsilon^{\rho}\frac{\delta^3}{(1+\delta^2)}[\bar{\Omega}_1\Omega_5 - \Omega_1 \bar{\Omega}_7] + \epsilon^{2\rho}\frac{3\delta^2}{2\nu_0(1+4\delta^2)(1+\delta^2)^2}\Omega_3(|\Omega_5|^2 + |\Omega_7|^2)\nonumber \\
        \frac{d}{d\tau} \Omega_5 &=& - \nu_0 \frac{1+\delta^2}{\delta^2} \Omega_5 -\epsilon^{(1-\rho+2\sigma)}\frac{\delta^2-1}{\epsilon\delta}\Omega_1 \Omega_3    \label{E:initialscale}\\
        && \qquad \qquad -\epsilon^{2\sigma}\Omega_5\left(\frac{\delta^6(3+\delta^2)}{2\nu_0(4+\delta^2)(1+\delta^2)}  |\Omega_1|^2+\frac{1+3\delta^2}{2\nu_0\delta^2(1+4\delta^2)(1+\delta^2)} |\Omega_3|^2\right) \nonumber    \\
        \frac{d}{d\tau} \Omega_7 &=& - \nu_0 \frac{1+\delta^2}{\delta^2} \Omega_7 +\epsilon^{(1-\rho+2\sigma)}\frac{\delta^2-1}{\epsilon\delta}\Omega_1 \bar{\Omega}_3 \nonumber \\
        && \qquad \qquad -\epsilon^{2\sigma}\Omega_7\left(\frac{\delta^6(3+\delta^2)}{2\nu_0(4+\delta^2)(1+\delta^2)}  |\Omega_1|^2 + \frac{1+3\delta^2}{2\nu_0\delta^2(1+4\delta^2)(1+\delta^2)} |\Omega_3|^2\right)\nonumber
\end{eqnarray}
Looking closely at the powers of $\epsilon$ appearing in \eqref{E:initialscale}, one notices that if the following restrictions on $\sigma$ and $\rho$ hold, then the leading order terms are of the desired form given by (\ref{E:leadingorder}):
\begin{equation}
    \begin{aligned}
        -\frac{1}{2}\leq\sigma<0, \quad \mbox{or }  \alpha-\frac{1}{2}&\leq\beta<\alpha \\
        0\leq\rho<1, \quad \mbox{or } \alpha\leq\phi&<\alpha+1. \nonumber
    \end{aligned}
\end{equation}

\begin{Remark}
As mentioned in the opening paragraph of this section, we aim to compute asymptotic expansions in $\epsilon$ for $\Omega_1$, $\Omega_3$, $\Omega_5$, and $\Omega_7$. If $\beta$ and $\phi$ satisfy the above inequalities, then the expansions computed have the expected properties for any value of $\alpha$.  No restrictions on $\alpha$ frees us from being constrained to particular relative values of the viscocity and aspect ratio of the domain. Namely, after connecting the viscocity and aspect ratio through $\nu=\epsilon^{\alpha}\nu_0$,  the value of $\alpha$ determines which of $\nu$ or $\epsilon$ is larger. Therefore the following results will hold for small values of $\nu$ and $\epsilon$ regardless of which is bigger relative to one another.
\end{Remark}

Proceeding with the computation of the asymptotic expansions, for simplicity, set $\phi=\alpha$ and $\beta=\alpha-\frac{1}{2}$ for any value of $\alpha>1/2$.  With this choice of parameter values, the final scaled system that we work with from this point on can be obtained by substituting $\delta=1+\epsilon_0\epsilon$ into \eqref{E:initialscale} and using the Taylor series expansions of the $\delta$ dependent coefficients. The resulting system is given below by \eqref{E:perturbation}.

\begin{equation}
    \begin{aligned}
        \frac{d}{d\tau}\Omega_1 &= \sum_{j\geq 0}\epsilon^j[-c^1_j\nu_0\Omega_1 + c^2_j(\Omega_3\Omega_7-\bar{\Omega}_3\Omega_5) + \frac{c^3_j}{\nu_0} \Omega_1(|\Omega_5|^2+|\Omega_7|^2)]  \\
        \frac{d}{d\tau}\Omega_3 &= \sum_{j\geq 0}\epsilon^j[-\nu_0\Omega_3 + c^4_j(\bar{\Omega}_1\Omega_5-\Omega_1\bar{\Omega}_7) + \frac{c^5_j}{\nu_0} \Omega_3(|\Omega_5|^2+|\Omega_7|^2)] \label{E:perturbation} \\
        \frac{d}{d\tau}\Omega_5 &= \sum_{j\geq 0}\epsilon^{j}[-\epsilon^{-1}\frac{1}{\nu_0}\Omega_5(c^{6}_j|\Omega_1|^2+c^{7}_j|\Omega_3|^2)-c^{8}_j\nu_0\Omega_5-c^{9}_{j+1}(\Omega_1\Omega_3) ]   \\
        \frac{d}{d\tau}\Omega_7 &= \sum_{j\geq 0}\epsilon^{j}[-\epsilon^{-1}\frac{1}{\nu_0}\Omega_7(c^{6}_j|\Omega_1|^2+c^{7}_j|\Omega_3|^2)-c^{8}_j\nu_0\Omega_7-c^{9}_{j+1}(\Omega_1\bar{\Omega}_3) ],
    \end{aligned}
\end{equation}
where
\[
\sum_{j\geq 0}c^{1}_j\epsilon^j = \frac{1}{\delta^2}, \quad \sum_{j\geq 0}c^{2}_j\epsilon^j = \frac{1}{\delta(1+\delta^2)}, \quad \sum_{j\geq 0}c^{3}_j\epsilon^j = \frac{3\delta^6}{2(4+\delta^2)(1+\delta^2)^2},\quad \sum_{j\geq 0}c^{4}_j\epsilon^j = \frac{\delta^3}{1+\delta^2},
\]
\[
 \sum_{j\geq 0}c^{5}_j\epsilon^j = \frac{3\delta^2}{2(1+4\delta^2)(1+\delta^2)^2}, \quad \sum_{j\geq 0}c^{6}_j\epsilon^j = \frac{\delta^6(3+\delta^2)}{2(4+\delta^2)(1+\delta^2)}, \quad \sum_{j\geq 0}c^{7}_j\epsilon^j = \frac{1+3\delta^2}{2\delta^2(1+4\delta^2)(1+\delta^2)},
 \]
 and
\[
 \quad \sum_{j\geq 0}c^{8}_j\epsilon^j =\frac{1+\delta^2}{\delta^2}, \quad \sum_{j\geq 0}c^{9}_j\epsilon^j = \frac{\delta^2-1}{\delta}.
 \]
With these scalings, it is evident that $\Omega_1$ and $\Omega_3$ are the slow variables, evolving with respect to $\tau$, to leading order, at an $\mathcal{O}(1)$ rate as $\epsilon \to 0$, while $\Omega_5$ and $\Omega_7$ evolve on the faster time scale $\mathcal{O}(\epsilon^{-1})$ as $\epsilon\to0$. This matches the time scale separation we saw present in the previous sections, which can be seen below by reversing the scalings in the leading order terms.

\begin{equation}
    \begin{aligned}
    \epsilon^{-\beta}\omega_{1,3} &= \Omega_{1,3} \sim  e^{-\nu_0\tau}     = e^{-(\epsilon^{-\alpha}\nu)(\epsilon^{\alpha}t)}= e^{-\nu t} \\
    \epsilon^{-\phi}\omega_{5,7} &= \Omega_{5,7} \sim                        e^{-\frac{\epsilon^{-1}}{\nu_0}(|\Omega_1(0)|^2+|\Omega_3(0)|^2)\tau} = e^{-\frac{\epsilon^{\alpha-1}}{\nu}\epsilon^{2\beta}(|\omega_1(0)|^2+|\omega_3(0)|^2)(\epsilon^{\alpha}t)}=e^{-\frac{1}{\nu}(|\omega_1(0)|^2+|\omega_3(0)|^2)t}. \nonumber
    \end{aligned}
\end{equation}

We now proceed using methods from geometric singular perturbation theory. Setting $\epsilon = 0$ in \eqref{E:perturbation} leads to the following leading order slow dynamics
\begin{equation*}
    \begin{aligned}
        \frac{d}{d\tau}\Omega_1 &=  -\nu_0\Omega_1 + \frac{1}{2}(\Omega_3\Omega_7-\bar{\Omega}_3\Omega_5) + \frac{3}{40\nu_0} \Omega_1(|\Omega_5|^2+|\Omega_7|^2) \\
        \frac{d}{d\tau}\Omega_3 &=  -\nu_0\Omega_3 + \frac{1}{2}(\bar{\Omega}_1\Omega_5-\Omega_1\bar{\Omega}_7) + \frac{3}{40\nu_0} \Omega_1(|\Omega_5|^2+|\Omega_7|^2) \\
        0 &= -\frac{1}{5\nu_0}\Omega_5(|\Omega_1|^2+|\Omega_3|^2) \label{E:reducedslow}  \\
        0 &= -\frac{1}{5\nu_0}\Omega_7(|\Omega_1|^2+|\Omega_3|^2),
\end{aligned}	
\end{equation*}
and so the leading order slow manifold is $M_0=\{\Omega_5=\Omega_7=0\}$. Observe that in the perturbed system \eqref{E:perturbation}, this manifold of fixed points for the $\epsilon=0$ reduced slow system is no longer invariant. This can be seen in the differential equations for $\Omega_5$ and $\Omega_7$. However, since $M_0$ is a normally hyperbolic manifold and the vector field in \eqref{E:perturbation} satisfies the smoothness conditions of Fenichel's theorems, a perturbed invariant manifold, $M_\epsilon$, exists for sufficiently small $\epsilon > 0$  and is $\mathcal{O}(\epsilon)$ close to $M_0$. See Theorem 9.1 in \cite{Fenichel1}. Any trajectory in phase space will approach this manifold exponentially fast and then track the slow dynamics on $M_{\epsilon}$.

Defining the fast variable $s=\tau/\epsilon$ and setting $\epsilon = 0$ in \eqref{E:perturbation}, we find the leading order fast dynamics to be given by
\begin{eqnarray*}
\frac{d}{ds}\Omega_1 &=& 0 \nonumber \\
\frac{d}{ds}\Omega_3 &=& 0   \label{E: reducedfast} \\
\frac{d}{ds}\Omega_5 &=&  -\frac{1}{5\nu_0}\Omega_5(|\Omega_1|^2+|\Omega_3|^2) \nonumber \\
\frac{d}{ds}\Omega_7 &=& -\frac{1}{5\nu_0}\Omega_7(|\Omega_1|^2+|\Omega_3|^2). \nonumber
\end{eqnarray*}
Assuming expansions of $\Omega_{i}(s)$ for $i=1,3,5,7$ to be of the form $\Omega_i(s)=\Omega_{i0}(s)+\epsilon\Omega_{i1}(s)+\mathcal{O}(\epsilon^2)$ away from the slow manifold, we find
\begin{eqnarray*}
\Omega_{10} &=& \Omega_{10}(0) \\
\Omega_{30} &=& \Omega_{30}(0)  \\
\Omega_{50} &=& \Omega_{50}(0) e^{-\frac{|\Omega_{10}(0)|^2+|\Omega_{30}(0)|^2}{5\nu_0}s} \\
\Omega_{70} &=& \Omega_{70}(0) e^{-\frac{|\Omega_{10}(0)|^2+|\Omega_{30}(0)|^2}{5\nu_0}s}.
\end{eqnarray*}
Here we see that away from the slow manifold, to leading order, the higher order modes are decaying at a rate $\mathcal{O}\left(e^{-\frac{|\Omega_{10}(0)|^2+|\Omega_{30}(0)|^2}{5\nu_0}s}\right)$, while the lowest modes are constant. This is consistent with the initial rapid decay rates among the higher modes seen in the previous sections, as seen in the calculation below:

\[
\epsilon^{-\phi}\omega_{50,70} = \Omega_{50,70} \sim e^{-\frac{|\Omega_{10}(0)|^2+|\Omega_{30}(0)|^2}{5\nu_0}s} = e^{-\frac{\epsilon^{2\beta}(|\omega_{10}(0)|^2+|\omega_{30}(0)|^2)}{5(\epsilon^{-\alpha}\nu)}(\epsilon^{\alpha-1}t)}=e^{-\frac{|\omega_{10}(0)|^2+|\omega_{30}(0)|^2}{5\nu}t}.
\]

The dynamics on the slow manifold will determine whether solutions evolve towards a bar state or a dipole. To analyze this, we consider again system \eqref{E:perturbation} and compute formal asymptotic expansions of the solutions in terms of the slow variable $\tau$. We begin by writing
\begin{equation}
    \begin{aligned}
        \Omega_1(\tau,\epsilon) &= \Omega_{10}(\tau)+\epsilon\Omega_{11}(\tau)+\epsilon^2\Omega_{12}(\tau)+\mathcal{O}(\epsilon^3)  \\
        \Omega_3(\tau,\epsilon) &= \Omega_{30}(\tau)+\epsilon\Omega_{31}(\tau)+\epsilon^2\Omega_{32}(\tau)+\mathcal{O}(\epsilon^3)  \\
        \Omega_5(\tau,\epsilon) &= \Omega_{50}(\tau)+\epsilon\Omega_{51}(\tau)+\epsilon^2\Omega_{52}(\tau)+\mathcal{O}(\epsilon^3) \label{E:expansion}\\
        \Omega_7(\tau,\epsilon) &= \Omega_{70}(\tau)+\epsilon\Omega_{71}(\tau)+\epsilon^2\Omega_{72}(\tau)+\mathcal{O}(\epsilon^3)
    \end{aligned}
\end{equation}
Next, we compute terms in these expansions up to and including $\mathcal{O}(\epsilon)$ terms. To do so, first consider the $\mathcal{O}(\epsilon^{-1})$ terms present in \eqref{E:perturbation} and match these terms with the derivatives taken in \eqref{E:expansion}. The only terms present are
\begin{equation*}
    \begin{aligned}
        0 &= -\frac{1}{5\nu_0}\Omega_{50}(|\Omega_{10}|^2+|\Omega_{30}|^2) \nonumber \\
        0 &= -\frac{1}{5\nu_0}\Omega_{70}(|\Omega_{10}|^2+|\Omega_{30}|^2) \nonumber
    \end{aligned}
\end{equation*}
and so we find $|\Omega_{50}|^2=|\Omega_{70}|^2=0$. Next, matching the $\mathcal{O}(1)$ terms, we obtain the following system governing the dynamics of the $\mathcal{O}(1)$ terms of $\Omega_1$ and $\Omega_3$, as well as algebraic equations that determine the $\mathcal{O}(\epsilon)$ terms of $\Omega_5$ and $\Omega_7$:
\begin{equation*}
    \begin{aligned}
        \frac{d}{d\tau}\Omega_{10} &= -\nu_0\Omega_{10}  \\
        \frac{d}{d\tau}\Omega_{30} &= -\nu_0\Omega_{30}  \\
        0 &= -\frac{1}{5\nu_0}\Omega_{51}(|\Omega_{10}|^2+|\Omega_{30}|^2)-2\epsilon_0\Omega_{10}\Omega_{30} \label{E:slowtimeODE0} \\
        0 &= -\frac{1}{5\nu_0}\Omega_{71}(|\Omega_{10}|^2+|\Omega_{30}|^2)+2\epsilon_0\Omega_{10}\bar{\Omega}_{30}
    \end{aligned}
\end{equation*}
Solving these we obtain the following expressions:

\begin{equation}
    \begin{aligned}
        \Omega_{10} &= \Omega_{10}(0)e^{-\nu_0\tau}, \qquad\qquad\qquad \Omega_{30} = \Omega_{30}(0)e^{-\nu_0\tau}  \\
        \Omega_{51} &= -\frac{10\nu_0\epsilon_0\Omega_{10}(0)\Omega_{30}(0)}{|\Omega_{10}(0)|^2+|\Omega_{30}(0)|^2}, \quad \Omega_{71} = \frac{10\nu_0\epsilon_0\Omega_{10}(0)\bar{\Omega}_{30}(0)}{|\Omega_{10}(0)|^2+|\Omega_{30}(0)|^2}  \label{E:slowtimeexpansion0}
    \end{aligned}
\end{equation}
Yet to be computed are the $\mathcal{O}(\epsilon)$ terms for the lower modes, $\Omega_1$ and $\Omega_3$.  The relevant equations are
\begin{equation}
    \begin{aligned}
        \frac{d}{d\tau}\Omega_{11} &= -\nu_0 \Omega^r_{11}+\frac{1}{2}(\Omega_{30}\Omega_{71}-\bar{\Omega}_{30}\Omega_{51})+2\nu_0\epsilon_0\Omega_{10}  \\
        \frac{d}{d\tau}\Omega_{31} &= -\nu_0 \Omega_{31}+\frac{1}{2}(\bar{\Omega}_{10}\Omega_{51}-\Omega_{10}\bar{\Omega}_{71}).  \label{E:slowtimeODE1}
    \end{aligned}
\end{equation}
Using \eqref{E:slowtimeexpansion0} to solve \eqref{E:slowtimeODE1} we obtain
\begin{equation}
    \begin{aligned}
        \Omega_{11} &= \Omega_{11}(0)e^{-\nu_0 \tau} +\nu_0\epsilon_0\tau e^{-\nu_0 \tau}\left[2\Omega_{10}(0)+\frac{10|\Omega_{30}(0)|^2\Omega_{10}(0)}{|\Omega_{10}(0)|^2+|\Omega_{30}(0)|^2}\right]  \\
        \Omega_{31} &= \Omega_{31}(0)e^{-\nu_0 \tau} -\nu_0\epsilon_0\tau e^{-\nu_0 \tau}\frac{10|\Omega_{10}(0)|^2\Omega_{30}(0)}{|\Omega_{10}(0)|^2+|\Omega_{30}(0)|^2}. \label{E:slowtimeexpansion1}
    \end{aligned}
\end{equation}

Together, equations \eqref{E:slowtimeexpansion0} and \eqref{E:slowtimeexpansion1} make the approximations to $\Omega_1$ and $\Omega_3$ up to and including $\mathcal{O}(\epsilon)$, which will be denoted by $\bar{\Omega}_1$ and $\bar{\Omega}_3$:
\begin{equation*}
\begin{aligned}
\bar{\Omega}_1(\tau) &:= \Omega_{10}(0)e^{-\nu_0\tau} + \epsilon \left(\Omega_{11}(0)e^{-\nu_0 \tau} +\nu_0\epsilon_0\tau e^{-\nu_0 \tau}\left[2\Omega_{10}(0)+\frac{10|\Omega_{30}(0)|^2|\Omega_{10}(0)|^2}{|\Omega_{10}(0)|^2+|\Omega_{30}(0)|^2}\right]\right)  \\\bar{\Omega}_3(\tau) &:= \Omega_{30}(0)e^{-\nu_0\tau} + \epsilon\left(\Omega_{31}(0)e^{-\nu_0 \tau} -\nu_0\epsilon_0\tau e^{-\nu_0 \tau}\frac{10|\Omega_{10}(0)|^2|\Omega_{30}(0)|^2}{|\Omega_{10}(0)|^2+|\Omega_{30}(0)|^2} \right). \label{E:ExpOrderEpsilon}
\end{aligned}
\end{equation*}
Of interest will be the magnitudes of the scaled low modes, which will be defined by $X:=|\Omega_1|^2$ and $Y:=|\Omega_3|^2$. In computing these, we obtain
\begin{equation}
\begin{aligned}
X(\tau) &= X_0(0)e^{-2\nu_0\tau} + \epsilon\left[ X_1(0)e^{-2\nu_0\tau}+\epsilon_0\nu_0\tau e^{-2\nu_0\tau}\left(4 X_0(0)+\frac{20X_0(0)Y_0(0)}{X_0(0)+Y_0(0)}\right)\right]  + \mathcal{O}(\epsilon^2)\\
Y(\tau) &= Y_0(0)e^{-2\nu_0\tau} + \epsilon\left[Y_1(0)e^{-2\nu_0\tau}-\frac{20X_0(0)Y_0(0)}{X_0(0)+Y_0(0)}\epsilon_0\nu_0\tau e^{-2\nu_0\tau}\right] + \mathcal{O}(\epsilon^2), \label{E:magnitudes}
\end{aligned}
\end{equation}
where we have used the notation $X(\tau)=X_0(\tau)+\epsilon X_1(\tau)+\mathcal{O}(\epsilon^2)$ and $Y(\tau)=Y_0(\tau)+\epsilon Y_1(\tau)+\mathcal{O}(\epsilon^2)$. For notational convenience, let $K_{x_0,y_0}=\frac{20X_0(0)Y_0(0)}{X_0(0)+Y_0(0)}$.

Now take $\bar{X}:=X_0+\epsilon X_1$ and $\bar{Y}:=Y_0+\epsilon Y_1$ to be the approximations to $|\Omega_1|^2$ and $|\Omega_3|^2$. For each $0<\epsilon\ll1$, there exists a finite interval of time on which these approximations are valid. Define $\tau_{+}$ and $\tau_{-}$ as follows
\begin{eqnarray}
        \tau_{+} &=& \frac{1}{\nu_0K_{x_0,y_0}}\left(\frac{Y_0(0)}{\epsilon}+Y_1(0)\right) \nonumber \\
        \tau_{-}&=& \frac{1}{\nu_0(K_{x_0,y_0}+4X_0(0))}\left(\frac{X_0(0)}{\epsilon}+X_1(0)\right) \label{E:critT}.
\end{eqnarray}
Recall that, as approximations to the nonnegative quantities $|\Omega_1|^2$ and $|\Omega_3|^2$,  $\bar{X}$ and $\bar{Y}$ must too be nonnegative. Observe that for $\epsilon_0=1$,   $\bar{X}(\tau_{+})=0$ and $\bar{X}(\tau)<0$ on $\tau>\tau_{+}$.  For $\epsilon_0=-1$, this property is shared by $\bar{Y}$ on $\tau\geq\tau_-$. These properties indicate that the approximations are certainly not valid for values of $\tau$ beyond $\tau_\pm.$

As a direct consequence of these observations, on a finite time interval $\epsilon_0$ will determine the bar state toward which the system evolves. This is summarized by Proposition \ref{perturbationbars} below and is consistent with the results in \S\ref{AsymmetricTorus}. The key property of the approximations $X(\tau)$ and $Y(\tau)$ in \eqref{E:magnitudes} is the linear growth exhibited by the $\mathcal{O}(\epsilon)$ terms. In particular, the opposite sign on the $\epsilon_0\nu_0\tau e^{-2\nu_0 \tau}$ terms allows for the evolution toward the correct bar state. This linear growth originates from the resonant forcing terms in the differential equations for $\Omega_{11}$ and $\Omega_{31}$, which can be seen after substituting the expressions for $\Omega_{10}$ and $\Omega_{30}$ given in \eqref{E:slowtimeexpansion0} into \eqref{E:slowtimeODE1}.

\begin{Proposition}
\label{perturbationbars}
Let $0<\epsilon\ll1$. Consider the approximations to $|\Omega_1|^2$ and $|\Omega_3|^2$ up to $\mathcal{O}(\epsilon)$ given by \eqref{E:magnitudes},
\begin{eqnarray*}
\bar{X}(\tau,\epsilon) &:=& X_0(\tau)+\epsilon X_1(\tau)  \\
\bar{Y}(\tau,\epsilon) &:=& Y_0(\tau)+\epsilon Y_1(\tau).
\end{eqnarray*}
There exists positive times $\tau_{+}$ and $\tau_{-}$, defined by \eqref{E:critT}, for which, when $\epsilon_0=1$, $\lim\limits_{\tau\rightarrow\tau_{+}}\frac{\bar{X}(\tau)}{\bar{Y}(\tau)}=\infty$, indicating evolution to an x-bar state, and, when $\epsilon_0=-1$, $\lim\limits_{\tau\rightarrow\tau_{-}}\frac{\bar{X}(\tau)}{\bar{Y}(\tau)}=0$, indicating evolution to a y-bar state. The critical times $\tau_{+}$ and $\tau_{-}$ are $\mathcal{O}(1/\epsilon)$ as $\epsilon\to0$.

\end{Proposition}
\begin{proof}
Using the expressions given by \eqref{E:magnitudes}, consider the ratio $\frac{\bar{X}}{\bar{Y}}$ as it is a measure of how close the system is to one bar state or another.
\begin{equation}
\begin{aligned}
\frac{\bar{X}(\tau)}{\bar{Y}(\tau)} &= \frac{X_0(0)e^{-2\nu_0 \tau} + \epsilon[ X_1(0)e^{-2\nu_0\tau} + \epsilon_0\nu_0(K_{x_0,y_0}+4X_0(0))\tau  e^{-2\nu_0\tau}]}{Y_0(0)e^{-2\nu_0 \tau} + \epsilon[ Y_1(0)e^{-2\nu_0\tau} - \epsilon_0\nu_0 K_{x_0,y_0}\tau e^{-2\nu_0\tau}]} \\
&= \frac{X_0(0)+\epsilon[X_1(0)+\epsilon_0\nu_0(K_{x_0,y_0}+4X_0(0))\tau]}{Y_0(0)+\epsilon[Y_1(0)-\epsilon_0\nu_0K_{x_0,y_0}\tau]} \label{E:ratio}
\end{aligned}
\end{equation}
Observe that, if $\epsilon_0=1$  the denominator of the ratio in \eqref{E:ratio} decreases monotonically and vanishes at $\tau_{+}= \frac{1}{\nu_0K_{x_0,y_0}}\left(\frac{Y_0(0)}{\epsilon}+Y_1(0)\right)$, and if $\epsilon_0=-1$ then the numerator decreases monotonically and vanishes at $\tau_{-}=\frac{1}{\nu_0(K_{x_0,y_0}+4X_0(0))}\left(\frac{X_0(0)}{\epsilon}+X_1(0)\right)$. Hence for $\epsilon_0=1$, we have $\lim\limits_{\tau\rightarrow\tau_{+}}\frac{\bar{X}(\tau)}{\bar{Y}(\tau)}=\infty$ (indicating an x-bar state) while for $\epsilon_0=-1$, we have $\lim\limits_{\tau\rightarrow\tau_{-}}\frac{\bar{X}(\tau)}{\bar{Y}(\tau)}=0$ (indicating a y-bar state).
\end{proof}

This proposition illustrates that as time $\tau$ increases in the finite interval on which the approximations in \eqref{E:magnitudes} are valid, the bar state toward which $\frac{\bar{X}}{\bar{Y}}$ tends is determined solely by $\epsilon_0$. The fact that $\tau^\pm = \mathcal{O}(1/\epsilon)$ implies that, for larger $\epsilon$, there will be more rapid convergence to the bar state. This could be related to the results of \cite{BouchetSimonnet09}, which suggest that $|\delta-1| = \epsilon$ needs to be sufficiently large before a bar state will dominate the dynamics under appropriate stochastic perturbations.

%%%%%%%%%%%%%%%%%%

\section{Acknowledgements}
The authors wish to thank the anonymous referees for their helpful comments. M.B was partially supported by National Science Foundation (NSF) grant DMS 1411460, and K.S. was partially supported by NSF DMS 1550918.

%%%%%%%%%%%%%%%%%

 \section{Appendix}\label{Appendix1}
\textit{Completion of proof of Theorem \ref{Symmetrictheorem}}\\
We have already showed the existence of a family of local invariant manifolds for \eqref{E:PQsystem}. Here we compute the form that this family takes. We begin by splitting \eqref{E:PQsystem} into real and imaginary parts to get the following 8 dimensional real system of ODEs

\begin{equation}
    \begin{aligned}
        \dot{R} &= (1+R)(P_{re}-Q_{re}) \\
        \dot{A} &= -2\nu A + \frac{3}{20\nu}A(w+z)  \\
        \dot{w} &= -4\nu w - \frac{2}{5\nu}wA  \\
        \dot{z} &= -4\nu z -\frac{2}{5\nu}zA \label{E:PQsplit}\\
        \dot{P}_r &= -2\nu P_{re} +\frac{z}{2}(1-R)-\frac{1}{5\nu}P_{re}A+(P_{re}-Q_{re})P_{re}+\frac{1}{2}P_{re}Q_{re}(1-\frac{1}{R})+\frac{1}{2}P_{im} Q_{im}(1+\frac{1}{R}) \\
        \dot{P}_i &= -2\nu P_{im} -\frac{1}{5\nu}P_{im}A+(P_{re}-Q_{re})P_{im}+\frac{1}{2}P_{im}Q_{re}(1-\frac{1}{R})-\frac{1}{2}P_{re}Q_{im}(1+\frac{1}{R})\\
        \dot{Q}_r &= -2\nu Q_{re} +\frac{w}{2}(R-1)-\frac{1}{5\nu}Q_{re}A+(P_{re}-Q_{re})Q_{re}+\frac{1}{2}P_{re}Q_{re}(\frac{1}{R}-1)-\frac{1}{2}P_{im}Q_{im}(\frac{1}{R}+1) \\
        \dot{Q}_i &= -2\nu Q_{im} -\frac{1}{5\nu}Q_{im}A+(P_{re}-Q_{re})Q_{im}+\frac{1}{2}P_{im}Q_{re}(\frac{1}{R}+1)+\frac{1}{2}P_{re}Q_{im}(\frac{1}{R}-1).
    \end{aligned}
\end{equation}
Observe that the above system has a line of equilibrium points $\vec{r}=(r,0,0,0,0,0,0,0)$. The Jacobian of \eqref{E:PQsplit} at each of these equilibrium points has 7 negative eigenvalues and 1 zero eigenvalue. Thus for each point on the line $\vec{r}$, there is an associated 7 local dimensional stable manifold, $W^s(\vec{r})$, and 1 dimensional center manifold, $W^{c}(\vec{r})$, which is in fact globally defined and corresponds to the line $\vec{r}$ itself. Below we will proceed to explicitly compute this family of stable manifolds, denoted $\mathcal{M}_r$, up to and including the quadradic terms. This is a long and tedious calculation but comes from standard invariant manifold theorems from ODE, see Theorem 4.1 in \cite{Chicone}. Since the linearization of \eqref{E:PQsplit} has 7 negative eigenvalues and 1 zero eigenvalue, this theorem guarantees the existence of an 8 dimensional center-stable manifold associated to each fixed point. As we can identify the line of equilibrium points, $\vec{r}$, as the globally defined center manifold, we indeed know that the following computation will result in the codemension 1 stable manifold, i.e. the center-stable manifold is the union of $M_r$ and $\vec{r}$. We begin the derivation by first shifting coordinates to move the equilibrium point at $(r,0,0,0,0,0,0,0)$ to the origin.
\[
\left( \begin{array}{c}
\tilde{R} \\
\tilde{A} \\
\tilde{w} \\
\tilde{z} \\
\tilde{P}_r \\
\tilde{P}_i \\
\tilde{Q}_r \\
\tilde{Q}_i \\
\end{array} \right)
=
\left( \begin{array}{c}
R-r \\
A \\
w \\
z \\
P_{re} \\
P_{im} \\
Q_{re} \\
Q_{im} \\
\end{array} \right).
\]

The resulting system with fixed point the fixed point at the origin corresponding to $\vec{r}$ is
\begin{equation}
    \begin{aligned}
        \dot{\tilde{R}} &= (1+\tilde{R}+r)(\tilde{P}_r-\tilde{Q}_r) \\
        \dot{\tilde{A}} &= -2\nu \tilde{A} + \frac{3}{20\nu}\tilde{A}(\tilde{w}+\tilde{z})\\
        \dot{\tilde{w}} &= -4\nu \tilde{w} - \frac{2}{5\nu}\tilde{w}\tilde{A} \label{E:PQshifted} \\
        \dot{\tilde{z}} &= -4\nu \tilde{z} -\frac{2}{5\nu}\tilde{z}\tilde{A}  \\
        \dot{\tilde{P}}_r &= -2\nu \tilde{P}_r +\frac{\tilde{z}}{2}(1-\tilde{R}-r)-\frac{1}{5\nu}\tilde{P}_r\tilde{A}+(\tilde{P}_r-\tilde{Q}_r)\tilde{P}_r+\frac{1}{2}\tilde{P}_r\tilde{Q}_r(1-\frac{1}{\tilde{R}+r})+\frac{1}{2}\tilde{P}_i\tilde{Q}_i(1+\frac{1}{\tilde{R}+r}) \\
        \dot{\tilde{P}}_i &= -2\nu \tilde{P}_i -\frac{1}{5\nu}\tilde{P}_i\tilde{A}+(\tilde{P}_r-\tilde{Q}_r)\tilde{P}_i+\frac{1}{2}\tilde{P}_i\tilde{Q}_r(1-\frac{1}{\tilde{R}+r})-\frac{1}{2}\tilde{P}_r\tilde{Q}_i(1+\frac{1}{\tilde{R}+r}) \\
        \dot{\tilde{Q}}_r &= -2\nu \tilde{Q}_r +\frac{\tilde{w}}{2}(\tilde{R}+r-1)-\frac{1}{5\nu}\tilde{Q}_r\tilde{A}+(\tilde{P}_r-\tilde{Q}_r)\tilde{Q}_r+\frac{1}{2}\tilde{P}_r\tilde{Q}_r(\frac{1}{\tilde{R}+r}-1)-\frac{1}{2}\tilde{P}_i\tilde{Q}_i(\frac{1}{\tilde{R}+r}+1)  \\
        \dot{\tilde{Q}}_i &= -2\nu \tilde{Q}_i -\frac{1}{5\nu}\tilde{Q}_i\tilde{A}+(\tilde{P}_r-\tilde{Q}_r)\tilde{Q}_i+\frac{1}{2}\tilde{P}_i\tilde{Q}_r(\frac{1}{\tilde{R}+r}+1)+\frac{1}{2}\tilde{P}_r\tilde{Q}_i(\frac{1}{\tilde{R}+r}-1).
    \end{aligned}
\end{equation}

The Jacobian of \eqref{E:PQshifted} at the origin is
\[ J(\vec{0})=
\left( \begin{array}{cccccccc}
0 & 0 & 0 & 0 & 1+r & 0 & -(1+r) & 0 \\
0 & -2\nu & 0 & 0 & 0 & 0 & 0 & 0 \\
0 & 0 & -4\nu & 0  & 0 & 0 & 0 & 0 \\
0 & 0 & 0 & -4\nu & 0 & 0 & 0 & 0 \\
0 & 0 & 0 & \frac{1}{2}(1-r) & -2\nu & 0 & 0 & 0 \\
0 & 0 & 0 & 0 & 0 & -2\nu & 0 & 0 \\
0 & 0 & \frac{1}{2}(r-1) & 0 & 0 & 0 & -2\nu & 0 \\
0 & 0 & 0 & 0 & 0 & 0 & 0 & -2\nu
\end{array} \right).
\]

Let us define
\[
\mathbb{X}=
\left( \begin{array}{c}
\tilde{R} \\
\tilde{A} \\
\tilde{w} \\
\tilde{z} \\
\tilde{P}_r \\
\tilde{P}_i \\
\tilde{Q}_r \\
\tilde{Q}_i
\end{array} \right), \quad J=J(\vec{0}),
\]
\[
N(\mathbb{X})=
\left( \begin{array}{c}
\tilde{R}(\tilde{P}_r-\tilde{Q}_r) \\
\frac{3}{20\nu}\tilde{A}(\tilde{w}+\tilde{z}) \\
-\frac{2}{5\nu}\tilde{w}\tilde{A} \\
-\frac{2}{5\nu}\tilde{z}\tilde{A} \\
-\frac{1}{2}\tilde{z}\tilde{R}-\frac{1}{5\nu}\tilde{A}\tilde{P}_r +(\tilde{P}_r-\tilde{Q}_r)\tilde{P}_r + \frac{1}{2}\tilde{P}_r\tilde{Q}_r(1-\frac{1}{\tilde{R}+r})+\frac{1}{2}\tilde{P}_i\tilde{Q}_i(1+\frac{1}{\tilde{R}+r}) \\
-\frac{1}{5\nu}\tilde{A}\tilde{P}_i +(\tilde{P}_r-\tilde{Q}_r)\tilde{P}_i + \frac{1}{2}\tilde{P}_i\tilde{Q}_r(1-\frac{1}{\tilde{R}+r})-\frac{1}{2}\tilde{P}_r\tilde{Q}_i(1+\frac{1}{\tilde{R}+r}) \\
\frac{1}{2}\tilde{w}\tilde{R}-\frac{1}{5\nu}\tilde{A}\tilde{Q}_r +(\tilde{P}_r-\tilde{Q}_r)\tilde{Q}_r + \frac{1}{2}\tilde{P}_r\tilde{Q}_r(\frac{1}{\tilde{R}+r}-1)-\frac{1}{2}\tilde{P}_i\tilde{Q}_i(1+\frac{1}{\tilde{R}+r}) \\
-\frac{1}{5\nu}\tilde{A}\tilde{Q}_i +(\tilde{P}_r-\tilde{Q}_r)\tilde{Q}_i + \frac{1}{2}\tilde{P}_i\tilde{Q}_r(1+\frac{1}{\tilde{R}+r})+\frac{1}{2}\tilde{P}_r\tilde{Q}_i(\frac{1}{\tilde{R}+r}-1) \\
\end{array} \right).
\]

Then we can write the system as
\[ \dot{\mathbb{X}}=J\mathbb{X}+N(\mathbb{X}). \]

Before computing the stable manifold we must change variables and diagonalize the matrix $J$. The matrix of eigenvectors of $J$ and its inverse are given by
\[
S=
\left( \begin{array}{cccccccc}
1 & 0 & \frac{r+1}{4\nu} & -\frac{r+1}{4\nu}&-\frac{r+1}{2\nu} & 0 & \frac{r+1}{2\nu} & 0 \\
0 & 1 & 0 & 0 & 0 & 0 & 0 & 0 \\
0 & 0 & -\frac{4\nu}{r-1} & 0 & 0 & 0 & 0 & 0 \\
0 & 0 & 0 & \frac{4\nu}{r-1} & 0 & 0 & 0 & 0 \\
0 & 0 & 0 & 1 & 1 & 0 & 0 & 0 \\
0 & 0 & 0 & 0 & 0 & 1 & 0 & 0 \\
0 & 0 & 1 & 0 & 0 & 0 & 1 & 0 \\
0 & 0 & 0 & 0 & 0 & 0 & 0 & 1 \\
\end{array} \right), \qquad
S^{-1}=
\left( \begin{array}{cccccccc}
1 & 0 & -\frac{r^2-1}{16\nu^2} & -\frac{r^2-1}{16\nu^2} &\frac{r+1}{2\nu} & 0 & -\frac{r+1}{2\nu} & 0 \\
0 & 1 & 0 & 0 & 0 & 0 & 0 & 0 \\
0 & 0 & -\frac{r-1}{4\nu} & 0 & 0 & 0 & 0 & 0 \\
0 & 0 & 0 & \frac{r-1}{4\nu} & 0 & 0 & 0 & 0 \\
0 & 0 & 0 & -\frac{r-1}{4\nu} & 1 & 0 & 0 & 0 \\
0 & 0 & 0 & 0 & 0 & 1 & 0 & 0 \\
0 & 0 & \frac{r-1}{4\nu} & 0 & 0 & 0 & 1 & 0 \\
0 & 0 & 0 & 0 & 0 & 0 & 0 & 1 \\
\end{array} \right)
\]

So we have
\[
\left( \begin{array}{cccccccc}
0 & 0 & 0 & 0 & 0 & 0 & 0 & 0\\
0 & -2\nu & 0 & 0 & 0 & 0 & 0 & 0 \\
0 & 0 & -4\nu & 0 & 0 & 0 & 0 & 0 \\
0 & 0 & 0 & -4\nu & 0 & 0 & 0 & 0\\
0 & 0 & 0 & 0 & -2\nu & 0 & 0 & 0 \\
0 & 0 & 0 & 0 & 0 & -2\nu & 0 & 0 \\
0 & 0 & 0 & 0 & 0 & 0 & -2\nu & 0 \\
0 & 0 & 0 & 0 & 0 & 0 & 0 & -2\nu
\end{array} \right)
= \Delta = S^{-1}JS.
\]

Defining
 \[\mathbb{Y}=S^{-1}\mathbb{X},\]
the dynamics of $\mathbb{Y}$ are given by
\begin{align*}
\dot{\mathbb{Y}} &= S^{-1}JS\mathbb{Y}+ S^{-1}N(S\mathbb{Y})= \Delta\mathbb{Y}+S^{-1}N(S\mathbb{Y}).
\end{align*}

So now we have
\[S
\left( \begin{array}{c}
y_1 \\
y_2 \\
y_3 \\
y_4 \\
y_5 \\
y_6 \\
y_7 \\
y_8
\end{array} \right)
=
\left( \begin{array}{c}
y_1 + \frac{r+1}{4\nu}y_3 - \frac{r+1}{4\nu}y_4 - \frac{r+1}{2\nu}y_5 + \frac{r+1}{2\nu}y_7 \\
y_2 \\
-\frac{4\nu}{r-1}y_3 \\
\frac{4\nu}{r-1}y_4 \\
y_4+y_5 \\
y_6 \\
y_3+y_7 \\
y_8
\end{array} \right).
\]

For notational convenience, let $\lambda=y_1 + \frac{r+1}{4\nu}y_3 - \frac{r+1}{4\nu}y_4 - \frac{r+1}{2\nu}y_5 + \frac{r+1}{2\nu}y_7$. Then, we have

$N(S\mathbb{Y})$=
\[
\left( \begin{array}{c}
\lambda(y_4+y_5-y_3-y_7) \\
\frac{3}{5(r-1)}y_2(y_4-y_3) \\
\frac{8}{5(r-1)}y_2y_3 \\
-\frac{8}{5(r-1)}y_2y_4 \\
-\frac{2\nu}{r-1}\lambda y_4-\frac{1}{5\nu}y_2(y_4+y_5) +(y_4+y_5-y_3-y_7)(y_4+y_5) +\frac{1}{2}(y_4+y_5)(y_3+y_7)(1-\frac{1}{\lambda+r})+\frac{1}{2}y_6y_8(1+\frac{1}{\lambda+r}) \\
-\frac{1}{5\nu}y_2y_6 +y_6(y_4+y_5-y_3-y_7)+\frac{1}{2}y_6(y_3+y_7)(1-\frac{1}{\lambda+r})-\frac{1}{2}y_8(y_4+y_5)(1+\frac{1}{\lambda+r})\\
-\frac{2\nu}{r-1}\lambda y_3-\frac{1}{5\nu}y_2(y_3+y_7) +(y_4+y_5-y_3-y_7)(y_3+y_7) +\frac{1}{2}(y_4+y_5)(y_3+y_7)(\frac{1}{\lambda+r}-1)-\frac{1}{2}y_6y_8(1+\frac{1}{\lambda+r}) \\
-\frac{1}{5\nu}y_2y_8 +y_8(y_4+y_5-y_3-y_7)+\frac{1}{2}y_6(y_3+y_7)(1+\frac{1}{\lambda+r})+\frac{1}{2}y_8(y_4+y_5)(\frac{1}{\lambda+r}-1)\\
\end{array} \right).
\]

After computing $S^{-1}N(S\mathbb{Y})$ we see that the complete system in the $\mathbb{Y}$ variables is
\begin{align}
        \dot{y}_1 &= \lambda(y_4+y_5-y_3-y_7) + \frac{r+1}{r-1}(y_3-y_4)\lambda + \frac{r+1}{10\nu^2}y_2(y_7-y_5) +\frac{r+1}{2\nu}(y_4+y_5-y_3-y_7)^2 \notag\\
        & \quad +\frac{r+1}{2\nu}(y_4+y_5)(y_3+y_7)(1-\frac{1}{\lambda+r})+\frac{r+1}{2\nu}y_6y_8(1+\frac{1}{\lambda+r})\notag  \\
        \dot{y}_2 &= -2\nu y_2 + \frac{3}{5(r-1)}y_2(y_4-y_3)  \notag\\
        \dot{y}_3 &= -4\nu y_3 - \frac{2}{5\nu}y_2y_3 \notag \\
        \dot{y}_4 &= -4\nu y_4 - \frac{2}{5\nu}y_2y_4  \notag\\
        \dot{y}_5 &= -2\nu y_5 + \frac{2}{5\nu}y_2y_4-\frac{2\nu}{r-1}\lambda y_4 -\frac{1}{5\nu}y_2(y_4+y_5)+(y_4+y_5-y_3-y_7)(y_4+y_5) \label{E:Ysystem} \\
        & \quad +\frac{1}{2}(y_4+y_5)(y_3+y_7)(1-\frac{1}{\lambda+r})+\frac{1}{2}y_6y_8(1+\frac{1}{\lambda+r})\notag \\
        \dot{y}_6 &= -2\nu y_6 -\frac{1}{5\nu}y_2y_6+y_6(y_4+y_5-y_3-y_7)+\frac{1}{2}y_6(y_3+y_7)(1-\frac{1}{\lambda+r})-\frac{1}{2}y_8(y_4+y_5)(1+\frac{1}{\lambda+r})\notag \\
        \dot{y}_7 &= -2\nu y_7 + \frac{2}{5\nu}y_2y_3-\frac{2\nu}{r-1}\lambda y_3 -\frac{1}{5\nu}y_2(y_3+y_7)+(y_4+y_5-y_3-y_7)(y_3+y_7) \notag \\
        & \quad  +\frac{1}{2}(y_4+y_5)(y_3+y_7)(\frac{1}{\lambda+r}-1)-\frac{1}{2}y_6y_8(1+\frac{1}{\lambda+r})\notag \\
        \dot{y}_8 &= -2\nu y_8 -\frac{1}{5\nu}y_2y_8+y_8(y_4+y_5-y_3-y_7)+\frac{1}{2}y_6(y_3+y_7)(1+\frac{1}{\lambda+r})+\frac{1}{2}y_8(y_4+y_5)(\frac{1}{\lambda+r}-1). \notag
\end{align}

Note that in grouping $\Theta(2)$ terms in the $\dot{y}_1$ equation in \eqref{E:Ysystem} we have:
\begin{align*}
\dot{y}_1 &= -\frac{r+1}{10\nu^2}y_2y_5 + \frac{r+1}{10\nu^2}y_2y_7-\frac{r+1}{2\nu}(\frac{r+1}{r-1}+\frac{1}{r})y_3y_4 + \frac{r+1}{2\nu}(\frac{1}{2}-\frac{r+1}{r-1}-\frac{1}{r})y_3y_5 \nonumber \\
& \quad +\frac{r+1}{2\nu}(\frac{1}{2}+\frac{r+1}{r-1})y_3y_7 +\frac{r(r+1)}{2\nu(r-1)}y_3^2 + \frac{r+1}{2\nu}(\frac{1}{2}+\frac{r+1}{r-1})y_4y_5 + \frac{r+1}{2\nu}(\frac{1}{2}-\frac{r+1}{r-1}-\frac{1}{r})y_4y_7 \nonumber \\
& \quad +\frac{r(r+1)}{2\nu(r-1)}y_4^2+\frac{r^2-1}{2\nu r}y_5y_7 + \frac{(r+1)^2}{2\nu r}y_6y_8 +\Theta(3). \nonumber
\end{align*}

Now we are ready to match coefficients on the stable manifold.  We know that the stable manifold takes the form $y_1=h(y_2,y_3,y_4,y_5,y_6,y_7,y_8)$ and goes through the origin with parabolic tangency. So let:
\[
  h(y_2,y_3,y_4,y_5,y_6,y_7,y_8)= \sum_{i=2}^{8}\sum_{j=i}^{8} c_{ij}y_iy_j + \Theta(3).
\]

Its derivative with respect to $t$ is:
\begin{align*}
\dot{h} &= \sum_{i=2}^{8} \frac{\partial h}{\partial y_i}\dot{y}_i  \\
&= -4\nu(c_{25}y_2y_5+c_{26}y_2y_6+c_{27}y_2y_7+c_{28}y_2y_8+c_{56}y_5y_6+c_{57}y_5y_7+c_{58}y_5y_8+c_{67}y_6y_7  \\
& \quad +c_{68}y_6y_8+c_{78}y_7y_8 +c_{22}y_2^2+c_{55}y_5^2+c_{66}y_6^2+c_{77}y_7^2+c_{88}y_8^2) -6\nu(c_{23}y_2y_3+c_{24}y_2y_4 \\
& \quad  +c_{35}y_3y_5+c_{36}y_3y_6+c_{37}y_3y_7+c_{38}y_3y_8+c_{45}y_4y_5+c_{46}y_4y_6+c_{47}y_4y_7+c_{48}y_4y_8)  \\
& \quad -8\nu(c_{34}y_3y_4+c_{33}y_3^2+c_{44}y_4^2) +\Theta(3). \nonumber
\end{align*}

Matching coefficients of among the quadratic terms in $\dot{h}$ and  $\dot{y}_1$, we obtain the nonzero coefficients below.
\[
c_{25} = \frac{r+1}{40\nu^3}, \quad
c_{27} = -\frac{r+1}{40\nu^3}, \quad
c_{34} = \frac{r+1}{16\nu^2}(\frac{r+1}{r-1}+\frac{1}{r}), \quad
c_{35} = -\frac{r+1}{12\nu^2}(\frac{1}{2}-\frac{r+1}{r-1}-\frac{1}{r}) , \quad
\]
 \[
c_{37} = -\frac{r+1}{12\nu^2}(\frac{1}{2}+\frac{r+1}{r-1}),\qquad\quad
c_{33} = -\frac{r(r+1)}{16\nu^2(r-1)} , \qquad\quad
c_{45} = -\frac{r+1}{12\nu^2}(\frac{1}{2}+\frac{r+1}{r-1})\qquad
\]
\[
c_{47} = -\frac{r+1}{12\nu^2}(\frac{1}{2}-\frac{r+1}{r-1}-\frac{1}{r}), \quad
c_{44} = -\frac{r(r+1)}{16\nu^2(r-1)}, \quad
c_{57} = -\frac{r^2-1}{8\nu^2 r} , \quad
c_{68} = -\frac{(r+1)^2}{8\nu^2 r}.
\]

The stable manifold for the origin in the $\mathbb{Y}$ variables becomes:
\begin{align*}
h(y_2,y_3,y_4,y_5,y_6,y_7,y_8) &= \frac{r+1}{40\nu^3}y_2y_5-\frac{r+1}{40\nu^3}y_2y_7+\frac{r+1}{16\nu^2}\left(\frac{r+1}{r-1}+\frac{1}{r}\right)y_3y_4 \\
& \quad -\frac{r+1}{12\nu^2}\left(\frac{1}{2}-\frac{r+1}{r-1}-\frac{1}{r}\right)y_3y_5-\frac{r+1}{12\nu^2}\left(\frac{1}{2}+\frac{r+1}{r-1}\right)y_3y_7-\frac{r(r+1)}{16\nu^2(r-1)}y_3^2 \\
& \quad -\frac{r+1}{12\nu^2}\left(\frac{1}{2}+\frac{r+1}{r-1}\right)y_4y_5-\frac{r+1}{12\nu^2}\left(\frac{1}{2}-\frac{r+1}{r-1}-\frac{1}{r}\right)y_4y_7-\frac{r(r+1)}{16\nu^2(r-1)}y_4^2  \\
& \quad -\frac{r^2-1}{8\nu^2 r} y_5y_7-\frac{(r+1)^2}{8\nu^2 r}y_6y_8 +\Theta(3).
\end{align*}

In the original variables, the stable manifold can be written as:

\begin{align}
R&=r+h\left(A,-\frac{r-1}{4\nu}w,\frac{r-1}{4\nu}z,P_{re}-\frac{r-1}{4\nu}z,P_{im},Q_{re}+\frac{r-1}{4\nu}w,Q_{im}\right) +\frac{r^2-1}{16\nu^2}w+\frac{r^2-1}{16\nu^2}z\nonumber\\
&\quad-\frac{r+1}{2\nu}P_{re}+\frac{r+1}{2\nu}Q_{re} +\Theta(3),\nonumber
\end{align}
or, equivalently,
\begin{align*}
R &= r+\frac{r^2-1}{16\nu^2}w +\frac{r^2-1}{16\nu^2}z - \frac{r+1}{2\nu}P_{re}+\frac{r+1}{2\nu}Q_{re}-\frac{r^2-1}{160\nu^4}Aw-\frac{r^2-1}{160\nu^4}Az +\frac{r+1}{40\nu^3}AP_{re}-\frac{r+1}{40\nu^3}AQ_{re}  \\
& \quad +\frac{r^2-1}{768\nu^4r}(7r^2+2r+1)wz -\frac{r+1}{96\nu^3r}(4r^2-r+1)wP_{re} + \frac{r+1}{96\nu^3}(3r+1)wQ_{re}+\frac{r^2-1}{768\nu^4}(3r+2)w^2  \\
& \quad -\frac{r+1}{96\nu^3}(3r+1)zP_{re}+\frac{r+1}{96\nu^3r}(4r^2-r+1)zQ_{re} +\frac{r^2-1}{768\nu^4}(3r+2)z^2 -\frac{r^2-1}{8\nu^2r}P_{re}Q_{re}\\
& \quad -\frac{(r+1)^2}{8\nu^2r}P_{im}Q_{im} + \Theta(3),
\end{align*}
concluding the proof of the theorem.

\bibliography{references}

\begin{thebibliography}{10}

\bibitem{Armbuster}
D.~Armbuster, B.~Nicolaenko, N.~Smaoui, and P.~Chossat.
\newblock Symmetries and dynamics for 2-d navier-stokes flow.
\newblock {\em Physica D: Nonlinear Phenomena}, 95(1):81--83, 1996.

\bibitem{BeckWayne13}
M.~Beck and C.~E. Wayne.
\newblock Metastability and rapid convergence to quasi-stationary bar states
  for the two-dimensional {N}avier--{S}tokes equations.
\newblock {\em Proc. Roy. Soc. Edinburgh Sect. A}, 143(5):905--927, 2013.

\bibitem{BouchetSimonnet09}
F.~Bouchet and E.~Simonnet.
\newblock Random changes of flow topology in two-dimensional and geophysical
  turbulence.
\newblock {\em Phys. Rev. Lett.}, 102(094504), 2009.

\bibitem{Chicone}
C.~Chicone.
\newblock {\em Ordinary Differential Equations with Applications}, volume~34 of
  {\em Texts in Applied Mathematics}.
\newblock Springer-Verlag New York, New York, second edition, 2006.

\bibitem{EMattingly01}
Weinan E and Jonathan~C. Mattingly.
\newblock Ergodicity for the {N}avier-{S}tokes equation with degenerate random
  forcing: finite-dimensional approximation.
\newblock {\em Comm. Pure Appl. Math.}, 54(11):1386--1402, 2001.

\bibitem{Fenichel1}
N.~Fenichel.
\newblock Geometric singular perturbation theory for ordinary differential
  equations.
\newblock {\em J. Diff. Eq.}, 31:53--98, 1979.

\bibitem{FoiasLuanSaut}
C.~Foias, L.~Hoang, and J.-C. Saut.
\newblock Asymptotic integration of navier-stokes equations with potential
  forces. ii. an explicit poincar\`{e}--dulac normal form.
\newblock {\em J. Funct. Anal.}, 260(10):3007---3035, 2011.

\bibitem{FoiasSaut2}
C.~Foias and J.-C. Saut.
\newblock Asymptotic behavior, as $t\rightarrow\infty$, of solutions of
  navier-stokes equations and nonlinear spectral manifolds.
\newblock {\em I. Indiana Univ. Math. J.}, 33(3):459--477, 1984.

\bibitem{FoiasSaut1}
C.~Foias and J.-C. Saut.
\newblock Asymptotic integration of navier-stokes equations with potential
  forces. i.
\newblock {\em I. Indiana Univ. Math. J.}, 40(1):305--320, 1991.

\bibitem{GalletYoung}
B.~Gallet and R.-W. Young.
\newblock A two-dimensional vortex condensate at high reynolds number.
\newblock {\em Journal of Fluid Mechanics}, 715:359--388, 2013.

\bibitem{Henry81}
D.~Henry.
\newblock {\em Geometric Theory of Semilinear Parabolic Equations}.
\newblock Springer-Verlag, Berlin, 1981.

\bibitem{IbrahimMaekawaMasmoudi17}
S.~Ibrahim, Y.~Maekawa, and N.~Masmoudi.
\newblock On pseudospectral bound for non-selfadjoint operators and its
  application to stability of kolmogorov flows.
\newblock {\em Preprint}, https://arxiv.org/pdf/1710.05132.pdf, 2017.

\bibitem{KimOkamodo}
S.-C. Kim and H.~Oka.
\newblock Unimodal patterns appearing in the kolmogorov flows at large reynolds
  numbers.
\newblock {\em Nonlinearity}, 28:3219, 2015.

\bibitem{MattinglyPardoux14}
J.~C. Mattingly and E.~Pardoux.
\newblock Invariant measure selection by noise: an example.
\newblock {\em Discrete and Continuous Dynamical Systems}, 34(10):4223--4257,
  2014.

\bibitem{MeshalkinSinaui61}
L.~D. Meshalkin and Ja.~G. Sina{\u\i}.
\newblock Investigation of the stability of a stationary solution of a system
  of equations for the plane movement of an incompressible viscous liquid.
\newblock {\em J. Appl. Math. Mech.}, 25:1700--1705, 1961.

\bibitem{Yin}
Montgomery~D.C. Yin, Z. and H.J.H. Clercx.
\newblock Alternative statistical-mechanical descriptions of decaying
  two-dimensional turbulence in terms of ``patches'' and ``points''.
\newblock {\em Phys. Fluids}, 15:1937--1953, 2003.

\bibitem{Zelik14}
S.~Zelik.
\newblock Inertial manifolds and finite-dimensional reduction for dissipative
  pdes.
\newblock {\em Proc. Roy. Soc. Edinburgh Sect. A}, 144(6):1245--1327, 2014.

\end{thebibliography}
\bibliographystyle{plain}
\end{document}